\documentclass[11pt, reqno]{amsart}
\usepackage[english]{babel}
\usepackage[utf8]{inputenc}
\usepackage[T1]{fontenc}
\usepackage{amsmath}
\usepackage{amsthm}
\usepackage{amssymb}
\usepackage{graphics}
\usepackage{xcolor}
\usepackage{enumitem}
\usepackage[pagebackref, colorlinks = true, linkcolor = teal, urlcolor  = teal, citecolor = violet]{hyperref}
\usepackage[margin=1in]{geometry}
\usepackage{cleveref}
\crefname{equation}{Eq.}{Eqs.}
\usepackage{bbm}
\usepackage{pifont}
\usepackage{tikz}
\usepackage{graphicx}
\usepackage{bbold}
\usepackage{appendix}
\usepackage{tikz}
\usepackage{pgfplots}
\usetikzlibrary{shapes.geometric, arrows}

\usepackage[normalem]{ulem}

\newtheorem{theorem}{Theorem}[section]
\newtheorem{definition}[theorem]{Definition}
\newtheorem{proposition}[theorem]{Proposition}
\newtheorem{corollary}[theorem]{Corollary}
\newtheorem{lemma}[theorem]{Lemma}
\newtheorem{remark}[theorem]{Remark}

\newcommand{\rank}{\operatorname{rank}}

\renewcommand{\d}{{\rm d}}

\newcommand{\pp}{{\mathbb P}}
\newcommand{\ee}{{\mathbb E}}

\newcommand{\rr}{{\mathbb R}}
\newcommand{\nn}{{\mathbb N}}
\newcommand{\cc}{{\mathbb C}}
\newcommand{\qq}{{\mathbb Q}}

\renewcommand{\P}{{\mathrm P}}

\newcommand{\mat}{\mathrm M_d(\cc)}
\newcommand{\matd}{M_d(\cc)}
\newcommand{\supp}{\operatorname{supp}}

\newcommand{\sta}{{\P(\cc^d)}}

\newcommand{\staalg}{\mathcal{B}}
\newcommand{\out}{\Omega}
\newcommand{\outalg}{\mathcal{O}}

\DeclareMathOperator{\linspan}{linspan}
\DeclareMathOperator{\spn}{span}
\renewcommand{\Re}{\operatorname{Re}}

\renewcommand{\Im}{\operatorname{Im}}

\makeatletter
\newcommand{\specialcell}[1]{\ifmeasuring@#1\else\omit$\displaystyle#1$\ignorespaces\fi}
\makeatother
\title{Analyticity of the pressure function for products of matrices}

\begin{document}
\author{Arnaud Hautecœur}
\email{arnaud.hautecoeur@math.univ-toulouse.fr}
\address{Institut de Mathématiques, UMR5219, Université de Toulouse, CNRS, UPS, F-31062 Toulouse Cedex 9, France}

\maketitle

\begin{abstract}
  The pressure function is a fundamental object in various areas of mathematics. Its regularity is studied to derive insights into phase transitions in certain physical systems or to determine the Hausdorff dimension of self-affine sets. In this paper, we prove the analyticity of the pressure function for products of non-invertible matrices satisfying an irreducibility and a contractivity assumptions. Additionally, we establish a variational principle for the pressure function, thereby generalizing previous results.
\end{abstract}

\tableofcontents
\newpage

\section{Introduction}
Let $d \in \nn$, we consider the complex vector space $\cc^d$ equipped with the euclidean norm $\| \cdot \|$ and $\mat$ the vector space of the complex matrices endowed with the euclidian norm $\|.\|$. Let $0<s_{-}<s_+$ and $\mu$ be a measure on $\mat$ such that for every $s_{-}\leq s \leq s_+$:
\begin{equation}
    \int_{\mat} \|v\|^s d \mu(v) < \infty.
\end{equation}
Let us denote by $I_\mu$ the interval $(s_{-},s_+)$. Let us define the \textit{pressure function} $\mathbf{P}: I_\mu \to \rr$ of $\mu$:
\begin{equation}\label{eq:defPression}
     \mathbf{P}(s)=\underset{n \to \infty}{\lim} \frac{1}{n} \log \left( \int_{\mat^n} \|v_n...v_1\|^s d \mu^{\otimes n}(v_1,...,v_n) \right).
\end{equation}
By submultiplicativity of the supremum norm and the integrability of $\mu$, this quantity is well defined by Fekete Lemma.

This quantity naturally appears in several fields under different names. The pressure is a key object in multifractal formalism. It enables the study of the multifractal structure of self-similar measures generated by iterated function systems (IFS) without overlap \cite{kaenmaki2017natural,falconer1988hausdorff,Fe03,Fe09}. In statistical mechanics, it can be seen as the \textit{free energy} of some classical spin chains. The analyticity of the pressure reflects the absence of phase transitions in classical spin chains. In this context, the computation of the pressure (or \textit{generalized Lyapunov exponent}) has also been investigated \cite{Va10}. It is deeply linked with the large-deviation regime of the norm of the random product of matrices $\|W_n\|$.

The regularity of the pressure function has been studied under several assumptions. Feng and Käenmäki studied the regularity of the pressure for a finite set of irreducible matrices $\mathbf{A}=(A_1,...,A_m)$ indexed by a finite alphabet $\mathcal{A}=\{1,...,m\}$. Let $(\Sigma, \theta)$ be the one-sided full shift over $\mathcal{A}$ \cite{Fe11}. In this case, the pressure is defined by
\begin{equation}
    \mathbf{P}(s) =\underset{n \to \infty}{\lim} \frac{1}{n} \log \left( \sum_{i_1,...i_n \in \mathcal{A}^n}   \|A_{i_n}...A_{i_1}\|^s  \right).
\end{equation}
They proved that the pressure function is differentiable for $s>0$ and satisfies the following variational principle:
\begin{equation}\label{eq:varprincipleFeng}
    \mathbf{P}(s)=\sup \{ s \mathcal{M}_*(\nu) + h(\nu) \; | \; \nu \in \mathcal{P}_\theta \},
\end{equation}
where $\mathcal{P}_\theta$ denotes the space of $\theta$-invariant probability measures on $\Sigma$, $h(\nu)$ denotes the the measure-theoretic entropy of $\nu$ with respect to $\theta$ and $\mathcal{M}_*(\nu)$ is the Lyapunov exponent of $\nu$ with respect to $\theta$ \cite[Equation 1.2]{Fe11}. Measures satisfying Equation \eqref{eq:varprincipleFeng} are called \textit{equilibrium states}. Ergodic properties of these equilibrium states have been extensively studied (See for example \cite{Fe11,morris2018ergodic}).

An alternative approach to studying the regularity of $\mathbf{P}$ is by employing spectral methods. This method appears naturally in the field of random walks on groups. For more information, one can check for instance \cite{benoist2016random, furstenberg1960products, xiao2021limit}. Under a strong irreducibility and a contractivity assumption, if $\supp \mu \subset GL_d(\cc)$, Y. Guivarc'h and E. Le Page proved that the pressure function defined in Equation \eqref{eq:defPression} is real-analytic on $I_\mu$ \cite{GLP15, GLP04}. For this purpose, they proved that $k(s)=\exp {P(s)}$ is an isolated eigenvalue of the following operator $\Gamma_s$:
$$\Gamma_s f(\hat{x}) = \int_{\mat} f( v \cdot \hat{x})\|vx\|^s d \mu(v), $$
where $\hat x$ is an element of the complex projective space $\sta$ and $f: \sta \to \cc$ is any bounded measurable function. This operator appears naturally to study random walks on linear groups. It has been investigated in the literature (see e.g. \cite{BL85,GLP15,xiao2023moderate}). Under Guivarc'h and Le Page assumptions, the operators $\Gamma_s$ are quasi-compact on a Banach space $(\mathcal{B}, \|.\|_{\mathcal{B}})$ and therefore admit a spectral gap. Moreover, they proved that the map $s \mapsto \Gamma_s$ is analytic on $I_\mu$ \cite{GLP04,GLP15}. Kato's perturbation theory \cite{kato2013perturbation} guarantees the analyticity of $\mathbf{P}$ on $I_\mu$.

In this paper, we prove the same result for non-invertible matrices verifying a weaker irreducibility assumption. These more general assumptions require a new proof of the quasi-compactness of the operators $\Gamma_s$. As in \cite{GLP15}, for every $s \in I_\mu$, we construct a Markov Operator $Q_s$ depending on $\Gamma_s$ by Doob's relativisation procedure. The quasi-compactness of the operators $Q_s$ implies the quasi-compactness of the operators $\Gamma_s$. The proof is based on Ionescu-Tulcea-Marinescu theorem \cite{ITM50,HeHe01}. Our main innovation is the method of proof of a Doeblin-Fortet inequality for $Q_s$ required in Ionescu-Tulcea-Marinescu theorem. More specifially we prove the existence of $n_0 \in \nn$ such that for any $f \in \staalg$:
\begin{equation}
    \|Q_s^{n_0} f \|_{\mathcal{B}} \leq R |f| + r^{n_0} \|f\|_{\mathcal B},
\end{equation}
where $R \geq 0$, $|.|$ is a semi-norm on $\mathcal B$ and $0 \leq r<\rho(Q_s)$ where $\rho(Q_s)$ denotes the spectral radius of $Q_s$ on $(\staalg, \|.\|_{\staalg})$. In these settings, $\rho(Q_s)=1$. Inspired by the strategy developed in \cite{BHP24}, we prove that the contractivity assumption implies the existence of a function $h : \nn \to \rr_+$ and a constant $C \geq 0$ such that: 
\begin{equation}
    \|Q_s^{n} f \|_{\mathcal{B}} \leq R |f| + C h(n) \|f\|_{\mathcal B},
\end{equation}
and 
\begin{equation}
    \underset{n \to \infty}{\lim} h(n)=0.
\end{equation}

Another goal of this paper is to extend the variational principle mentioned in Equation \eqref{eq:varprincipleFeng} to arbitrary Polish space alphabets.  As mentioned before, ergodic properties of the equilibrium states are objects of interest. In 2021, M. Piraino constructed an equilibrium state by using a spectral approach \cite{piraino2020weak}. It allowed him to derive interesting ergodic and statistical properties on the equilibrium states by applying Guivarc'h and Le Page results. Moreover, we establish a connection between this spectral approach and the thermodynamic formalism. 

Let $\mathcal{A}$ an arbitrary Polish space and $\Sigma=\mathcal A ^{\nn}$ endowed with the left-shift $\theta$. In this paper, we construct a probability measure $\qq^s$ on $\Sigma$ which is an equilibrium state for a sub-additive variational principle. We prove that $\qq^s$ is ergodic and has exponential decay of correlations for Hölder continuous functions.

The paper is organized as follows.
The first part of the paper is devoted to the proof of the analyticity of the pressure function. 
In Section \ref{sec:Results}, we discuss the assumptions and state the main results of the paper. 
In Section \ref{sec:Notations}, we define properly the operators $\Gamma_s$ and introduce the Banach space of $\alpha$-Hölder functions $(C^\alpha(\sta,\cc),\|.\|_\alpha)$ for an $\alpha$ well chosen.
In Section \ref{Sec:Construction} we construct Markov Operators $Q_s$ by using Doob's relative procedure. The construction is closely aligned with that described in \cite{GLP15}, albeit with necessary adaptations.
Section \ref{SecQuasiCompact} is devoted to the proof of quasi-compactness for the operators $\Gamma_s$ and $Q_s$. In this section, we construct explicitely the probability measure $\qq^s$ and the function $h$ mentioned before. 
In Section \ref{Sec:Analyticks}, we finally prove the analyticity of the pressure function on $I_\mu$. It requires us to prove that the map $s \mapsto \Gamma_s$ is analytic on $I_\mu$. The result then follows from Kato's perturbation theory.

The second part of the paper is devoted the thermodynamic formalism.
We prove that the probability measure $\qq^s$ has an exponential decay of correlation in Section \ref{sec:Decay} and satisfies a sub-additive variational principle in Section \ref{Sec:VariationalPrinciple}. 

Section \ref{sec:Examples} is devoted to examples. We show that the irreducibility is crucial to obtain the analyticity of the pressure function. Moreover, the Keep-Switch example shows that one can not expect the analyticity in $0$.

\section{Main Results}\label{sec:Results}

In this section we state the main results of this article and under which assumptions they hold.

\subsection{Assumptions}

The first assumption is an irreducibility one.

\medskip
\noindent{\bf (Irr)} There does not exist non-trivial subspaces $F \subset \cc^d$ such that, for every $v \in \supp \mu$, $v F \subset F$.
\medskip

This assumption is weaker than the usual strong irreducibility assumption used in the context of product of random matrices (see \cite{BL85,GLP15}). We recall that $\supp \mu$ is strongly irreducible if there does not exist a non-trivial finite union of proper subspaces invariant by $\supp \mu$.

The second one is a contractivity assumption. We define the set $T_0=\{I_d\}$ and for $n \in \nn$, 
\begin{equation}
    T_n = \{ v_n...v_1 \, | \, v_i \in \supp \mu \}.
\end{equation}
$T_n$ is the set of products of at most $n$ matrices in $\supp \mu$ and the identity. The set $T = \cup_n T_n$ contains all the finite products of the matrices of $\supp \mu$ and the identity. It is the closed semigroup generated by $\supp \mu$.

\medskip
\noindent{\bf (Cont)}  There exists a sequence $(v_k) \in T$ such that $\underset{k \to \infty}{\lim} \frac{v_k}{\|v_k\|} = v_\infty$ where $v_\infty$ is a rank-one matrix.
\medskip
In \cite{GLP04}, it is called proximality.

\subsection{The analyticity of the function $s \mapsto \mathbf{P}(s)$}

The main result of this paper is:

\begin{theorem}\label{Th:Analyticks}
Let $0<s_{-}<s_+$ and $\mu$ be a measure on $\mat$ such that:
\begin{enumerate}
    \item $\mu$ verifies the assumptions {\bf (Cont)} and {\bf (Irr)},
    \item for every $s_{-}\leq s \leq s_+$:
\begin{equation*}
    \int_{\mat} \|v\|^s d \mu(v) < \infty.
\end{equation*}
\end{enumerate}

Then, the pressure function $\mathbf{P}: I_\mu \to \rr$ of $\mu$:
\begin{equation*}
     \mathbf{P}(s)=\underset{n \to \infty}{\lim} \frac{1}{n} \log \left( \int_{\mat^n} \|v_n...v_1\|^s d \mu^{\otimes n}(v_1,...,v_n) \right)
\end{equation*}
is analytic on $(s_-,s_+)$.
\end{theorem}
If $\supp \mu \subset GL_d(\cc)$ satisfies the strong irreducibility and contractivity assumptions, the analyticity of $k$ was established by Guivarch and Le Page in \cite{GLP15}.

The key innovation of this paper lies in eliminating the need for the invertibility assumption. By introducing a novel proof technique, we demonstrate how this assumption can be successfully relaxed.

\section{Notations and Sketch of proof} \label{sec:Notations}
Let us now describe the strategy of proof. For $s \in I_\mu$ fixed, $k(s)$ is in fact the spectral radius and an eigenvalue of the operator $\Gamma_s$. In order to study the function $k$, we study the spectral properties of the operators $\Gamma_s$. We introduce properly these operators.
\subsection{Notations}

We consider the projective space $\sta$ equipped with its Borel $\sigma$-algebra $\staalg$. 
For a nonzero vector $x\in\mathbb C^d$, we denote $\hat x$ the corresponding equivalence class of $x$ in $\sta$. For $\hat x \in \sta$, $x$ denotes a representative of norm $1$.
For a linear map  $v\in \mat$ we denote by $v\cdot \hat{x} $ the element of the projective space represented by $v\,x$ whenever $v\,x\neq 0$. We equip the projective space $\sta$ with the metric:
\begin{equation}\label{def:DistanceProj}
    d(\hat{x},\hat{y}) = \sqrt{1-|\langle x,y \rangle |^2},
\end{equation}
where $x,y$ are unit representative vectors. For any continuous function $f: \sta \to \cc$, $v \in \mat$ and $x \in \cc$, we fix $f(v \cdot \hat x)\|vx\|=0$ whenever $vx=0$.
If $M \subset \sta$, we denote by $\P(M)$ the projective span of $M$. It is the smallest projective subspace containing $M$.

We consider the space of infinite sequences $\out:=\mathrm{M}_k(\cc)^{\mathbb N}$, write $\omega = (v_1,v_2, \dots)$ for any such infinite sequence, and denote by $\pi_n$ the canonical projection on the  first $n$ components, $\pi_n(\omega)=(v_1,\ldots,v_n)$. Let $\mathcal M$ be the Borel $\sigma$-algebra on $\mat$. For $n\in\nn$, let $\outalg_n$ be the  $\sigma$-algebra on $\Omega$ generated by the $n$-cylinder sets, i.e.\ $\outalg_n = \pi_n^{-1}(\mathcal M^{\otimes n})$. We equip the space $\Omega$ with the smallest $\sigma$-algebra $\outalg$ containing $\outalg_n$ for all $n\in \nn$. We identify $O_n\in \mathcal M^{\otimes n}$ with $\pi_n^{-1}(O_n)$, a function $f$ on $\mathcal M^{\otimes n}$  with $f \circ \pi_n$ and a measure $\mu^{\otimes n}$ with the measure $\mu^{\otimes n} \circ \pi_n$.
For $i \in \nn$, we define the random variable $V_i : \Omega \mapsto \mat$ as follows: for every $\omega=(v_1,v_2,...) \in \out$,
\begin{equation}
    V_i(\omega)=v_i.
\end{equation}
For $n \in \nn$, we define the random variable $W_n : \Omega \mapsto \mat$ as follows:
\begin{equation}
    W_n=V_n...V_1.
\end{equation}
Let $\mu$ be a measure on $\mat$ and let $s_{-} \in \rr_+^*$ and $s_+ \in \rr_+^* \cup \{ + \infty \}$ such that $s_{-}<s_+$. We denote by $I_\mu$ the interval $(s_{-},s_+)$ and we assume that for every $s \in [s_{-},s_+]$:
\begin{equation}
\label{eq:HypConvergence}
\int_{\mat} \|v\|^s\,\d\mu(v)<\infty.
\end{equation}
We fix $\alpha := \min \{ \frac{s_{-}}{3}, 1 \}$. The choice of this value will be explained later. We denote the space of $\alpha$-Hölder continuous functions on $\sta$ by $C^{\alpha}(\sta,\cc)$. For $f \in C^\alpha(\sta, \cc)$ we note,
$$\| f \|_{\infty} = \sup \{ |f(\hat x)| : \hat x \in \sta \}, \; \mbox{and } m_\alpha(f) = \sup \left \{ \frac{|f(\hat x)-f(\hat y)|}{\d( \hat x,\hat y)^{\alpha}} : \hat x, \hat y \in \sta, \hat x \neq  \hat y \right \}.  $$
The space $C^\alpha(\sta, \mathbb{C})$, equipped with the norm $\| \cdot \|_\alpha = \| \cdot \|_{\infty} + m_\alpha(\cdot)$, forms a Banach space. We denote the spectral radius of an endomorphism $L$ on $C^\alpha(\sta,\mathbb{C})$ as $\rho_\alpha(L)$. It is defined by:
\begin{equation*}
    \rho_\alpha(L) = \underset{n \to \infty}{\lim} \|L^n\|_\alpha^{1/n}.
\end{equation*}

For $s \in I_\mu$, we define the operator $\Gamma_s$ acting on the space $C^\alpha(\sta, \cc)$ : for $f \in C^\alpha(\sta,\cc)$ and $\hat{x} \in \sta$
$$\Gamma_s f(\hat{x}) = \int_{\mat} f( v \cdot \hat{x})\|vx\|^s d \mu(v). $$
The space $\sta$ is compact, therefore $f$ is bounded on $\sta$, it follows that $|f( v \cdot \hat{x})|\|vx\|^s$ is uniformly bounded by $\|f\|_\infty \|v\|^s$ ensuring this integral is well defined. 
If $z=s+it$ where $s \in I_\mu$, we define the operator $\Gamma_z$ as follows: for $f \in C^\alpha(\sta,\cc)$ and $\hat{x} \in \sta$
$$\Gamma_z f(\hat{x}) = \int_{\mat} f( v \cdot \hat{x}) e^{z \log \|vx\|} d \mu(v). $$
We denote by $\mathcal{P}(\sta)$ the set of probability measures on $\sta$. If $\sigma \in \mathcal{P}(\sta)$, for every $s \in I_\mu$, $\sigma \Gamma_s$ denotes the measure on $\sta$ such that for every continuous function $\varphi: \sta \to \cc$,
$$ \sigma \Gamma_s(\varphi)=\int_{\sta} \Gamma_s \varphi (\hat x) d \sigma(\hat x). $$
If $f : \sta \to \cc$ is a continuous function and $\nu \in \mathcal{P}(\sta) $, $f \nu$ denotes the measure defined by:
$$ f \nu(\varphi)=\int_{\sta} \varphi(\hat x) f(\hat x) d\nu( \hat x),$$
for any bounded measurable function $\varphi : \sta \to \cc$.
We define the function $k: I_\mu \to \rr_+$ as follows:
\begin{equation}
\label{Definitionks}
k(s)= \underset{n \to \infty}{\lim} \left [ \int_{\mat^n} \|v_n ... v_1\|^s d \mu^{\otimes n}(v_1,...,v_n) \right]^{\frac{1}{n}}
\end{equation}
Since the sequence $\left(\int_{\mat^n} \|v_n ... v_1\|^s d \mu^{\otimes n}(v_1,...,v_n)\right)$ is submultiplicative, this limit exists by Fekete lemma.

\subsection{Sketch of proof of the analyticity}
The method to prove the analyticity of the function $k$ relies on the quasi-compactness of the operators $\Gamma_s$ and on the analyticity of the map $z \mapsto \Gamma_z$. Indeed, for every $s \in I_\mu$, $k(s)$ is an isolated and simple eigenvalue of $\Gamma_s$ by the spectral gap property of $\Gamma_s$, and by using Theorem VII-1.8 \cite{kato2013perturbation}, the analyticity of $k$ follows. 

We do not directly prove the quasi-compactness of the operators $\Gamma_s$ but we construct Markov operators $Q_s$ by Doob's relativisation procedure. Then, we show that $Q_s$ admits a unique invariant probability measure and a spectral gap. In order to construct these operators, we need this theorem.

\begin{theorem}\label{Constru}
Assume {\bf (Cont)} and {\bf (Irr)} hold. For every $s \in I_\mu$, there exist a unique function $e_s :\sta \to \rr_+$ and a unique probability measure $\sigma$ on $\sta$ such that:
\begin{enumerate}
    \item $\sigma \Gamma_s  = k(s) \sigma$,
    \item $\Gamma_s e_s = k(s) e_s$,
    \item $\sigma(e_s)=1$.
\end{enumerate}
Moreoever, $e_s$ is strictly positive and $\bar{s}$-Hölder where $\bar{s}=\min \{1,\frac{s}{2}\}$.
\end{theorem}
The proof of this theorem is the goal of Section \ref{Sec:Construction}. The next step is the quasi-compactness of the operators $Q_s$. This property implies a spectral gap for the operators and ensures that $k(s)$ is an isolated eigenvalue of $\Gamma_s$.
The proof of the following theorem is the object of Section \ref{sec:ProofQuasiCompact}.

\begin{theorem}\label{QuasiCompactGammaZ}
Assume {\bf (Cont)} and {\bf (Irr)} hold. For every $s \in I_\mu$, the operator $\Gamma_s$ is a quasi-compact operator on $(C^\alpha(\sta,\cc)$. There exist two subspaces $\mathcal{F}_s$ and $\mathcal{G}_s$ of $C^\alpha(\sta,\cc)$ such that 
    \begin{enumerate}
        \item \label{Compact1} $C^\alpha(\sta,\cc) = \mathcal{F}_s \oplus \mathcal{G}_s $,
        \item \label{Compact1bis} $\Gamma_s \mathcal F_s\subset \mathcal F_s$ and $\Gamma_s\mathcal G_s\subset \mathcal G_s$,
        \item \label{Compact3}$\rho_\alpha(\Gamma_s|_{\mathcal{G}_s}) < \rho_\alpha(\Gamma_s)=k(s) $,
        \item \label{Compact2}$\mathcal{F}_s$ is finite dimensional and the spectrum of $\Gamma_s|_{\mathcal{F}_s}$ denoted by $\mathrm{Spec}(\Gamma_s|_{\mathcal{F}_s})$ is a finite subgroup of $U(1)=\{z\in\cc,\vert z\vert=1\}$ given by:
        $$\mathrm{Spec}(\Gamma_s|_{\mathcal{F}_s})=\{ e^{i\frac{l}{m}2\pi} \; | \; l=0,...,m-1 \},$$
        where $m \geq 1$. In addition, all these eigenvalues are simple.
    \end{enumerate}
\end{theorem}
Finally the proof of the analyticity of the function $k$ is detailed in Section \ref{Sec:Analyticks} and requires the analyticity of the map $s \mapsto \Gamma_s$.

\section{Construction of Markov Operators $Q_s$}\label{Sec:Construction}
The goal of this section is to prove Theorem \ref{Constru}.
The proof of the theorem is done in two steps. Firstly, we show the existence of a probability measure $\sigma$ and a strictly positive function $e_s$ such that $\sigma \Gamma_s  = k(s) \sigma$ and $\Gamma_s e_s = k(s) e_s$. By Doob's relativisation procedure, we properly define operators $Q_s$ as follows: for every $f \in C^\alpha(\sta,\cc)$ and $\hat{x} \in \sta$:
\begin{equation}\label{DefQs}
    Q_s f(\hat{x}) = \frac{1}{k(s)}\int_{\mat} f(v \cdot \hat{x}) e_s(v \cdot \hat{x}) \|vx\|^s d \mu(v).
\end{equation}
Secondly, we show that this Markov operator $Q_s$ has a unique invariant probability measure $\eta^s$. It ensures that $\Gamma_s$ admits a unique probability measure $\sigma$ such that $\sigma \Gamma_s = k(s) \sigma$.

\subsection{$k(s)$ is an eigenvalue of $\Gamma_s$.}
The goal of this section is to prove the following proposition which allows us to construct the Markov operators $Q_s$ \eqref{DefQs}.

\begin{proposition}\label{Prop:ExistenceProbaFonction}
For every $s \in I_\mu$, there exists a probability measure $\sigma$ and positive $\Bar{s}$-Hölder function $e_s$ such that $\sigma \Gamma_s = k(s) \sigma$ and $\Gamma_s e_s = k(s) e_s$ where $\Bar{s}=\min \{1,\frac{s}{2}\}.$   
\end{proposition}

We first recall a lemma from \cite{GLP15}.

\begin{lemma}[Lemma 2.7 \cite{GLP15}]\label{IneqSigma}
Let $\sigma$ be a probability measure on $\sta$ such that $\supp \sigma$ is not included in a hyperplane. Then there exists $c_s(\sigma)>0$ such that for every $v \in \mat$,

\begin{equation}\label{Eq.Sigma}
    \int_{\sta} \|vx\|^s \d \sigma(\hat{x}) \geq c_s(\sigma) \|v\|^s.
\end{equation}
\end{lemma}

We now prove the existence of a measure $\sigma$ satisfying $\sigma \Gamma_s = k(s) \sigma$.

\begin{lemma}\label{Lemma:ExistenceSigma}
Assume {\bf (Irr)} holds. For every $s \in I_\mu$, there exists a probability measure $\sigma \in \mathcal{P}(\sta)$ such that $\sigma \Gamma_s = k(s) \sigma$. Moreover, for any $\sigma$ verifying the previous equality, $\supp \sigma$ is not included in a hyperplane and:
\begin{align*}
    \int_{\mat^n} \|W_n\|^s d \mu^{\otimes n} \geq k(s)^n \geq c_s(\sigma) \int_{\mat^n} \|W_n\|^s d \mu^{\otimes n},
\end{align*}
where $c_s(\sigma)$ is defined in Lemma \ref{IneqSigma}.
\end{lemma}

\begin{proof}[Proof of Lemma \ref{Lemma:ExistenceSigma}.]
Let $s \in I_\mu$, we construct the map $\widetilde{\Gamma}_s$ acting on the space of probability measures on $\sta$ defined as follows: for $\nu \in \mathcal{P}(\sta)$ $\widetilde{\Gamma}_s\sigma=\frac{\nu\Gamma_s}{\nu \Gamma_s(1)}$, where $1$ denotes the constant function equal to $1$ on $\sta$.
We first prove that this operator is well defined by proving that for every $\nu \in \mathcal{P}(\sta)$, $\nu\Gamma_s(1)>0$. Indeed, if it was not the case for a probability measure $\nu$ on $\sta$, we would have:
$$\nu \Gamma_s(1)=\int_{\sta} \int_{\mat} \|vx\|^s d \mu(v) d \nu(\hat x)=0.$$
Then, for $\nu$ almost every $\hat x \in \sta$:
$$ \int_{\mat} \|vx\|^s d \mu(v)=0.$$
Then for $\mu$-almost every $v$, $vx=0$. It implies that $H= \linspan \{ x \, | \, \hat x \in \supp \sigma \}$ is invariant by $\supp \mu$. Since, $\supp \nu \neq \emptyset$, {\bf (Irr)} implies that $H=\cc^d$.
Moreover, $H \subset \cap_{v \in \supp \mu} \ker v$, it implies that for every $v \in \supp \mu$, $\ker v = \cc^d$ implying that $v=0$. We obtain a contradiction and then for every probability measure $\nu$ on $\sta$, $\nu \Gamma_s(1)>0$.
Now, we prove the existence of $\sigma$. Since $\int_{\mat} \|v\|^s d \mu(v) < \infty$, the operator $\widetilde{\Gamma}_s$ is continuous in the weak topology. The space $\mathcal{P}(\sta)$ is a compact convex space, by Schauder-Tychonof fixed point theorem there exists a probability measure $\sigma$ on $\sta$ such that $\widetilde{\Gamma}_s \sigma=\sigma$. It follows that $\sigma \Gamma_s = (\sigma \Gamma_s (1))\sigma$. For any continuous function $\phi$ on $\sta$, we obtain that:
\begin{equation*}
    k \int_{\sta} \phi(\hat{x}) d \sigma(\hat{x}) = \int_{\sta} \int_{\mat} \phi(v \cdot \hat{x}) \|vx\|^s d \mu(v) d \sigma (\hat{x}),
\end{equation*}
where $k=\sigma \Gamma_s (1)$. This equation implies that for every $\hat{x} \in \supp \sigma$, 
\begin{equation}\label{eq:Alternative}
    \mbox{either } vx=0 \, \mbox{or } v \cdot \hat{x} \in \supp \sigma \quad \mu \mbox{-a.e.}
\end{equation}
Let us recall that
$$H= \linspan \{ x \, | \, \hat x \in \supp \sigma \}.$$
Equation \eqref{eq:Alternative} implies that $H$ is a nonempty $\supp \mu -$invariant subspace of $\cc^d$. By Assumption {\bf (Irr)}, $H=\cc^d$. It follows that the projective space generated by $\supp \sigma$ is $\sta$, and we obtain that $\supp \sigma$ is not included in a hyperplane. 

Now let us prove that $k=k(s)$, by Lemma \ref{IneqSigma}, there exists $c_s(\sigma)>0$ such that for any $v \in \mat$:
\begin{equation*}
        \int_{\mat} \|vx\|^s \d \sigma(\hat{x}) \geq c_s(\sigma) \|v\|^s.
\end{equation*}
Since $\sigma \Gamma_s^n=k^n \sigma$, it follows that $$k^n=\int_{\mat^n} \int_{\sta} \|v_n...v_1x\|^s d\mu^{\otimes n}(v_1,...,v_n) d\sigma(\hat x) $$ and:
\begin{equation*}
    c_s(\sigma) \int_{\mat^n}\|W_n\|^s d\mu^{\otimes n}  \leq k^n \leq \int_{\mat^n}\|W_n\|^s d\mu^{\otimes n}.
\end{equation*}
It follows that $k=\underset{n \to \infty}{\lim}\left (\int \|W_n\|^s d\mu^{\otimes n}\right )^{1/n}=k(s)$.
\end{proof}

\begin{remark}\label{Remark:Positivityk}
Let us remark that this lemma also shows that for every $s \in I_\mu$, $k(s)>0$. It is crucial to derive the analyticity of $\mathbf{P} := \log k$.
\end{remark}

\begin{proposition}\label{Prop:Existencee_s}
Assume {\bf (Irr)} holds. For every $s \in I_\mu$, there exists a strictly positive $e_s: \sta \to \rr_+$ such that $\Gamma_s e_s = k(s) e_s$. Moreover this function is $\bar{s}$-Hölder where $\bar{s}=\min \{1,\frac{s}{2}\}$.
\end{proposition}
\begin{proof}
Let $s \in I_\mu$. By Lemma \ref{Lemma:ExistenceSigma}, there exists a probability measure $\sigma$ such that $\sigma \Gamma_s=k(s) \sigma$. We consider the sequence $(\Gamma_s^n 1)$. Let $\hat{x}, \hat{y} \in \sta$, by Lemma A.2 \cite{BHP24}, there exists $C \geq 0$ such that:
\begin{equation}\label{eq:Gammas}
    | \Gamma_s^n (1) (\hat{x}) - \Gamma_s^n (1) (\hat{y}) | \leq C \int_{\mat^n} \|W_n\|^s d \mu^{\otimes n} d(\hat{x},\hat{y})^{\Bar{s}},
\end{equation}
where $\Bar{s}=\min \{1,\frac{s}{2}\}$. Lemma \ref{Lemma:ExistenceSigma} implies that
\begin{equation*}
    \frac{1}{k(s)^n} \int_{\mat^n} \|W_n\|^s d \mu^{\otimes n} \leq \frac{1}{c_s(\sigma)}.
\end{equation*}
Combining this equation with Equation \eqref{eq:Gammas} implies that:
\begin{equation*}
    \frac{1}{k(s)^n} \left| \Gamma_s^n (1) (\hat{x}) - \Gamma_s^n (1) (\hat{y}) \right| \leq \frac{C}{c_s(\sigma)} d(\hat{x},\hat{y})^{\Bar{s}}.
\end{equation*}
It follows that the family $\left(\frac{\Gamma_s^n (1)}{k(s)^n}\right)$ is equicontinuous and bounded. Let $f_n : \sta \to \rr_+$ be the function defined as follows: for $\hat{x} \in \sta$:
\begin{equation*}
    f_n(\hat{x}) = \frac{1}{n} \sum_{m=1}^n \frac{\Gamma_s^m 1 (\hat{x})}{k(s)^m}.
\end{equation*}
By Ascoli Theorem, there exists a subsequence $(f_{n_l})_l$ converging towards a non-negative function $f : \sta \to \rr_+$ such that for every $\hat{x}, \hat{y} \in \sta$:
\begin{equation}
    | f(\hat{x})- f(\hat{y}) | \leq C d(\hat{x},\hat{y})^{\Bar{s}}.
\end{equation}
Moreover, for every $l \in \nn$ and $\hat{x} \in \sta$:
\begin{equation*}
    \Gamma_s f_{n_l}(\hat{x})= \frac{1}{{n_l}}  \sum_{m=1}^{n_l} \frac{\Gamma_s^{m+1} (1) (\hat{x})}{k(s)^m} = k(s) f_{n_l}(\hat{x}) + \frac{1}{n_l} \left [\frac{\Gamma_s^{n_l+1}1(\hat{x})}{k(s)^{n_l}} - \Gamma_s (1) (\hat{x})  \right].
\end{equation*}
By taking the limit when $l$ goes to infinity, it follows that $\Gamma_s f = k(s) f$.
Now, by contradiction, let assume that the function $f$ is not strictly positive and let 
$$M_0= \{ \hat{z} \in \sta \, | \, f(\hat{z})=0 \}$$
and
$$ H_0 = \linspan \{  z \, | \, \hat{z} \in M_0 \}.$$
By assumption, $M_0$ is not empty and there exists $\hat{x}$ such that $f(\hat{x})=0$.
The equation $\Gamma_s f(\hat{x}) = k(s) f(\hat{x})$ and the positivity of $f$ implies that for $\mu$-almost every $v \in \mat$, either $\|vx\|=0$ or $f(v \cdot \hat{x})=0$. 
Indeed, 
$$ 0=\Gamma_s f (\hat x) = \int_{\mat} f(v \cdot \hat x) \|vx\|^s d \mu(v).$$
It implies that $\mu$-almost everywhere, $f(v \cdot \hat x) \|vx\|^s=0$.
It follows that $H_0$ is a non-empty $\supp \mu$-invariant subspace of $\cc^d$. The assumption {\bf (Irr)} implies that $H_0 = \cc^d$, it follows that $\mathrm P(M_0)= \mathrm P(\cc^d)$ and $M_0$ is not included in a hyperplane. Let $\nu$ be a probability measure on $\sta$ supported on $M_0$. Lemma \ref{IneqSigma} implies the existence of $c_s(\nu)>0$ such that for every $v \in \mat$:
\begin{equation*}
    \int_{\sta} \|vx\|^s d \nu(\hat{x}) \geq c_s(\nu) \|v\|^s.
\end{equation*}
Then, for every $m \in \nn$:
\begin{equation*}
    \nu \left (\frac{\Gamma_s^m 1}{k(s)^m}\right) \geq c_s(\nu) \frac{\int_{\mat^m} \|W_m\|^s d \mu^{\otimes m}}{k(s)^m} \geq c_s(\nu).
\end{equation*}
It implies that for every $n \in \nn$:
\begin{equation*}
    \nu(f_n) \geq c_s(\nu).
\end{equation*}
By Lebesgue dominated convergence theorem, it follows that
\begin{equation*}
    \nu(f) \geq c_s(\nu)>0.
\end{equation*}
However, the support of $\nu$ is $M_0$, it implies that $\nu(f)=0$. We obtain a contradiction and it implies that the function $f$ is strictly positive.
\end{proof}
These last results allow us to construct the operators $Q_s$. Indeed, we can define the following functions $q^s_n$: 
\begin{equation}\label{eq:Defqs}
    q^s_n(\hat{x},A)=\frac{e_s(A \cdot \hat{x})}{e_s(\hat{x})} \frac{\|Ax\|^s}{k^n(s)},
\end{equation}
for every $n \in \nn^*$, $A \in \mat$ and $\hat{x} \in \sta$.
By construction, for every $n \in \nn^*$ and $\hat{x} \in \sta$, $\int q^s_n(\hat{x},W_n) d \mu^{\otimes n}=1$ and for any $A,B \in \mat$, by a direct computation one can show that:
$$ q^s_{n+1}( \hat x, AB) = q^s_n(B \cdot \hat x, A) \times q^s_1 ( \hat x, A).$$
It follows that the family $(q_n^s(\hat{x}, \cdot) d\mu^{\otimes n})$ is a consistent family of probability measures. Then, for every $\hat{x} \in \sta$, by Kolmogorov extension theorem, there exists a unique probability measure $\qq^s_{\hat{x}}$ on $\matd^{\nn}$ with marginals $q_n^s(\hat{x}, \cdot) \mu^{\otimes n}$. 
We can now define the Markov operators $Q_s$: for every $f \in C^\alpha(\sta,\cc)$ and $\hat{x} \in \sta$,
\begin{equation}\label{Def:Qs}
    Q_s f(\hat{x}) =  \int_{\matd} f(v \cdot \hat{x}) q_1(\hat{x}, v) d \mu(v).
\end{equation}
The iterates of $Q_s$ are given by:
$$Q_s^n f(\hat{x}) =  \int_{\matd^n} f(W_n \cdot \hat{x}) q^s_n(\hat{x}, W_n) d \mu^{\otimes n}(v_1,...,v_n).$$ 

\subsection{Uniqueness of the invariant measure of the operators $Q_s$.}\label{Unique}
The goal of this subsection is to prove that the operators $Q_s$ admit a unique invariant probability measure on $\sta$. 

\begin{proposition}\label{Prop:UniciteMesureInvariante}
Assume {\bf (Cont)} and {\bf (Irr)} hold. Then, for every $s \in I_\mu$, the operator $Q_s$ admits a unique invariant probability measure $\eta^s$ on $\sta$.
\end{proposition}
The proof of Proposition \ref{Prop:UniciteMesureInvariante} is based on the following lemma.
\begin{lemma}[Lemma 4.24 \cite{GLP04}]\label{lemma:Raugi}
Let $X$ a a compact metric space, $Q$ a Markov operator preserving the space of continuous functions $C(X)$. Let assume that for every $\varphi \in C(X)$, the sequence $(Q^n \varphi)$ is equicontinuous and that the only $Q$-invariant continuous functions are constant. Then $Q$ has a unique $Q$-invariant probability measure.
\end{lemma}
To prove Proposition \ref{Prop:UniciteMesureInvariante}, we then show that for any continuous function $\varphi$ on $\sta$, the sequence $(Q_s^n \varphi)$ is equicontinuous and that all the $Q_s$-invariant continuous functions are constant. It implies the uniqueness of the invariant measure of $Q_s$. For the first step, it suffices to show the equicontinuity for Hölder functions and we obtain the result on continuous functions by density. Before we prove this result, let us recall some results on the exterior products. 
For $x_1, x_2$ in $\cc^d$ we denote by $x_1\wedge x_2$ the alternating bilinear form $(y_1, y_2) \mapsto \det\big(\langle x_i, y_j\rangle \big)_{i,j=1}^2$. Then, the set of all $x_1\wedge x_2 $ is a generating family for the set $\wedge^2\cc^d$ of alternating bilinear forms on $\cc^d$, and we can define a Hermitian inner product by 
\[\langle x_1\wedge x_2, y_1\wedge y_2\rangle = \det\big(\langle x_i, y_j\rangle \big)_{i,j=1}^2, \]
and denote by $\|x_1\wedge x_2\|$ the associated norm. For any $\hat x, \hat y \in \sta$, a direct computation leads to:
\begin{equation}
    d(\hat x, \hat y)=\|x \wedge y\|.
\end{equation}
For a linear map $A$ on $\mathbb C^d$, we write $\wedge^2 A$ for the linear map on $\wedge^2\cc^d$ defined by
\begin{equation}\label{eq_defwedgepA}
\wedge^2 A \,(x_1\wedge x_2)=Ax_1\wedge Ax_2.
\end{equation}
Its operator norm will be denoted as $\|\wedge^2 A\|$.
For $s \in I_\mu$, we define the function $h_\alpha : \nn^* \to \rr_+$ such that for every $n \in \nn^*$:
\begin{equation}\label{def:halphan}
    h_\alpha(n)= \frac{1}{k(s)^n} \int_{\mat^n} \| \wedge^2 v_n...v_1 \|^\alpha \|v_n...v_1\|^{s-2\alpha} d \mu^{\otimes n}(v_1,...v_n).
\end{equation}
By Lemma \ref{Lemma:ExistenceSigma}, the sequence $(h_\alpha(n))$ is uniformly bounded by $\frac{1}{c_s(\sigma)}$. In order to prove the equicontinuity, it suffices to show the following lemma.

\begin{lemma}\label{Lemma:EstimationNormeHölder}
For every $s \in I_\mu$, there exist $A_s\geq 0$ and $B_s \geq 0$ such that for every $\varphi \in C^\alpha(\sta,\cc)$,
\begin{equation}
    \|Q_s^n \varphi\|_\alpha \leq A_s \|\varphi\|_\infty + B_s h_\alpha(n) \|\varphi\|_\alpha.
\end{equation}
\end{lemma}

\begin{remark}
Let us point out that if the sequence $(h_\alpha(n))$ converges towards $0$, we obtain a Doeblin-Fortet inequality unlocking the quasi-compactness of the operators $Q_s$. It is the object of Section \ref{sec:halpha}.
\end{remark}

Now, let us start the proof of Lemma \ref{Lemma:EstimationNormeHölder}. We first revisit Lemma 5.2 of \cite{BHP24}.

\begin{lemma}
    \label{lem:vxf(vx) holder}
    Let $z\in \cc$ be such that $\Re(z)>0$ and let $f\in C^\alpha(\sta)$. Then for any matrix $A\in \mat$, the function $f_A:\hat x\mapsto e^{z\log \|Ax\|}f(A\cdot \hat x)$ and $f_A(\hat x)=0$ whenever $Ax=0$, is $\alpha$-Hölder continuous with
    $$m_\alpha(f_A)\leq 2^{\alpha^*}\frac{|z/2|}{\alpha^*}\left [\|A\|^{\Re(z)} \|f\|_\infty +  \|A\|^{\Re(z)-2\alpha}\|\wedge^2 A\|^\alpha m_\alpha(f)\right],$$
    with $\alpha^*=\min\{1,\Re(z/2)\}$.
\end{lemma}

\begin{proof}
Let $\hat x,\hat y\in \sta$ be distinct. We use the shorthand $t^z$ for $e^{z\log t}$. Without loss of generality we can assume $\|Ax\|\geq \|Ay\|$.
    \begin{align}\label{eq:split f_A}
        |f_A(\hat x)-f_A(\hat y)|\leq \left|(\|Ax\|^z-\|Ay\|^z)f(A\cdot \hat x)\right|+\|Ay\|^{\Re(z)}\left|f(A\cdot \hat y)-f(A\cdot \hat x)\right|
    \end{align}
    where $0\times f=0$ even when the argument of $f$ may be undefined.

    We apply Lemmas A.1 and A.2 from \cite{BHP24} with $z$ set to $z/2$ and $K=[0,\|A\|^2]$ in Lemma A.2 \cite{BHP24},
    $$|\|Ax\|^z-\|Ay\|^z|\leq 2^{\alpha^*}\frac{|z/2|}{\alpha^*}\|A\|^{\Re(z)}d(\hat x,\hat y)^{\alpha^*}.$$
    Since $\alpha\leq \alpha^*$ and $d(\hat x,\hat y)\leq 1$, 
    \begin{align}
        \label{eq:bound (Axz-Ayz)f(Ax)}
        \left|(\|Ax\|^z-\|Ay\|^z)f(A\cdot \hat x)\right|\leq 2^{\alpha^*}\frac{|z/2|}{\alpha^*}\|A\|^{\Re(z)} \|f\|_\infty d(\hat x,\hat y)^\alpha.
    \end{align}

    For the second term on the right hand side of Equation~\eqref{eq:split f_A}, we can assume $\|Ay\|>0$ and therefore $\|Ax\|>0$. It follows from $d(A\cdot \hat x,A\cdot \hat y)=\frac{\|\wedge^2 A\ x\wedge y\|}{\|Ax\|\|Ay\|}$ that,
    $$\|Ay\|^{\Re(z)}\left|f(A\cdot \hat y)-f(A\cdot \hat x)\right|\leq \|Ay\|^{\Re(z)}m_\alpha(f) \frac{\|\wedge^2 A\|^\alpha}{\|Ax\|^\alpha\|Ay\|^\alpha}\|x\wedge y\|^\alpha$$
    Since $\|Ax\|\geq \|Ay\|$, $t\mapsto t^\alpha$ is non decreasing and $\|x\wedge y\|=d(\hat x,\hat y)$,
    $$\|Ay\|^{\Re(z)}\left|f(A\cdot \hat y)-f(A\cdot \hat x)\right|\leq \|Ay\|^{\Re(z)-2\alpha}\|\wedge^2 A\|^\alpha m_\alpha(f)d(\hat x,\hat y)^\alpha.$$
    By definition of $\alpha$, $\Re(z)-2\alpha\geq 0$, thus $t\mapsto t^{\Re(z)-2\alpha}$ is non decreasing. Hence,
    \begin{align}
        \label{eq:bound Ayz(f(Ay)-f(Ax))}
        \|Ay\|^{\Re(z)}\left|f(A\cdot \hat y)-f(A\cdot \hat x)\right|\leq \|A\|^{\Re(z)-2\alpha}\|\wedge^2 A\|^\alpha m_\alpha(f) d(\hat x,\hat y)^\alpha.
    \end{align}
    It follows from Equations~\eqref{eq:bound (Axz-Ayz)f(Ax)} and \eqref{eq:bound Ayz(f(Ay)-f(Ax))} that:
    $$m_\alpha(f_A)\leq 2^{\alpha^*}\frac{|z/2|}{\alpha^*}\left [\|A\|^{\Re(z)} \|f\|_\infty +  \|A\|^{\Re(z)-2\alpha}\|\wedge^2 A\|^\alpha m_\alpha(f)\right].$$
\end{proof}

\begin{lemma}\label{lemma:phiqsnbound}
Let $s \in I_\mu$. For every $\varphi \in C^\alpha(\sta,\cc)$, $n \in \nn$ and $A \in \mat$, the function $\hat{x} \mapsto \varphi(A \cdot \hat{x})q_n^s(\hat{x},A)$ where $\varphi(A \cdot \hat{x})q_n^s(\hat{x},A)=0$ whenever $Ax=0$, is $\alpha$-Hölder continuous. Moreover, there exist $a_s \geq 0$ and $b_s \geq 0$ only depending on $e_s$ and $\alpha$ such that:
    $$m_\alpha(\hat x \mapsto \varphi(A \cdot \hat x) q_n^s(\hat x,A))\leq \frac{1}{k(s)^n} \left [a_s \|\varphi\|_\infty \|A\|^s + b_s \|\varphi\|_\alpha \|\wedge^2 A\|^\alpha \|A\|^{s-2\alpha} \right ].$$
\end{lemma}

\begin{proof}
Let $s \in I_\mu$. Let $\varphi \in C^\alpha(\sta,\cc)$, $\hat{x},\hat{y} \in \sta$, and $A \in \mat$
\begin{align*}
    k(s)^n [\varphi(A \cdot \hat{x}) q_n^s(\hat{x},A)&-\varphi(A \cdot \hat{y}) q_n^s(\hat{y},A)]\\  =& \left [ \frac{1}{e_s(\hat{x})}- \frac{1}{e_s(\hat{y})} \right] e_s(A \cdot  \hat{x})\varphi(A \cdot \hat{x}) \|Ax\|^s \\
    &+ \frac{1}{e_s(\hat{y})}[e_s(A \cdot  \hat{x})\varphi(A \cdot \hat{x}) \|Ax\|^s - e_s(A \cdot  \hat{y})\varphi(A \cdot \hat{y}) \|Ay\|^s]
\end{align*}
Since $e_s$ is $\bar{s}$-Hölder and strictly positive, the function $\frac{1}{e_s}$ is also $\bar{s}$-Hölder and the first term can be estimated as follows:
\begin{equation}
    \left | \left [ \frac{1}{e_s(\hat{x})}- \frac{1}{e_s(\hat{y})} \right] e_s(A \cdot  \hat{x})\varphi(A \cdot \hat{x}) \|Ax\|^s\right | \leq  m_s(\tfrac{1}{e_s}) \|e_s\|_\infty\|\varphi\|_\infty\|A\|^s d(\hat{x},\hat{y})^{\bar{s}}.
\end{equation}
For the second term, it suffices to notice that $e_s$ is also $\alpha$-Hölder continuous because $\alpha \leq \bar{s}$. Applying Lemma \ref{lem:vxf(vx) holder} to $\bar{s}$, $\alpha$ and $f=e_s \varphi$ gives the following bound: 
\begin{align*} 
         &\left| \frac{1}{e_s(\hat{y})}[e_s(A \cdot  \hat{x})\varphi(A \cdot \hat{x}) \|Ax\|^s - e_s(A \cdot  \hat{y})\varphi(A \cdot \hat{y}) \|Ay\|^s] \right| \\
        \leq & \, c 2^{\alpha^*} \frac{s}{2 \alpha^*} \|\wedge^2 A\|^\alpha \|A\|^{s-2\alpha} \|\varphi\|_\alpha \|e_s\|_\alpha d(\hat{x},\hat{y})^\alpha \\
        &+ c 2^{\alpha^*} \frac{s}{2 \alpha^*} \|A\|^s \|\varphi\|_\infty \|e_s\|_\infty d(\hat{x},\hat{y})^\alpha
\end{align*}
where $c = \underset{\hat{z} \in \sta}{\sup}\frac{1}{e_s(\hat{z})}$. The lemma holds with $a_s=m_s(\frac{1}{e_s}) \|e_s\|_\infty +  c 2^{\alpha^*} \frac{s}{2 \alpha^*} \|e_s\|_\infty $ and $b_s= c 2^{\alpha^*} \frac{s}{2 \alpha^*}\|e_s\|_\alpha$.
\end{proof}
We can now prove Lemma \ref{Lemma:EstimationNormeHölder}.

\begin{proof}[Proof of Lemma \ref{Lemma:EstimationNormeHölder}]
Let $s \in I_\mu$. Let $\varphi \in C^\alpha(\sta,\cc)$, $\hat{x},\hat{y} \in \sta$. We have:
\begin{align*}
    |Q_s^n \varphi(\hat{x}) - Q_s^n \varphi(\hat{y})| \leq \int_{\mat^n} |\varphi(W_n \cdot \hat{x})q_n^s(\hat{x},W_n)-\varphi(W_n \cdot \hat{y})q_n^s(\hat{y},W_n)| d\mu^{\otimes n}.
\end{align*}
Lemma \ref{lemma:phiqsnbound} implies that there exists $a_s \geq 0$ and $b_s \geq 0$ such that:
\begin{equation}
    m_\alpha(Q_s^n \varphi) \leq a_s \|\varphi\|_\infty \frac{1}{k(s)^n} \int_{\mat^n} \|W_n\|^s d\mu^{\otimes n} + b_s \|\varphi\|_\alpha \frac{1}{k(s)^n} \int_{\mat^n} \|\wedge^2 W_n\|^\alpha \|W_n\|^{s-2\alpha} d \mu^{\otimes n}.
\end{equation}
Lemma \ref{Lemma:ExistenceSigma} gives the following bound:
\begin{equation}
    m_\alpha(Q_s^n \varphi) \leq a_s \frac{1}{c_s(\sigma)}\|\varphi\|_\infty  + b_s h_\alpha(n) \|\varphi\|_\alpha.
\end{equation}
$Q_s$ is a Markov operator, therefore $\|Q_s^n \varphi\| \leq \|\varphi\|_\infty$. Fixing $A_s=\frac{a_s}{c_s(\sigma)}+1$ and $B_s=b_s$, we obtain the lemma.
\end{proof}

\begin{lemma}\label{lem:Constant}
Assume {\bf (Irr)} and {\bf (Cont)} hold. Let $s \in I_\mu$, all the $Q_s$-invariant continuous functions are constant.
\end{lemma}
Before proving this lemma, let us state a standard lemma. 
\begin{lemma}\label{lemma:vndot}
Let $\hat{x} \in \sta$ and $(v_n) \in \mat^{\nn}$ converging to $v_\infty$ when $n$ goes to infinity such that $\|v_\infty x\|>0$, then the sequence $(v_n \cdot \hat{x})$ converges to $v_\infty \cdot \hat{x}$.
\end{lemma}
We do not detail the proof but this lemma is useful for the following proofs.
\begin{proof}[Proof of Lemma \ref{lem:Constant}]
Let $\varphi$ be a real-valued continuous function on $\sta$ satisfying the equation $Q_s \varphi = \varphi$. Since $\sta$ is a compact space and $\varphi$ a continuous function on $\sta$, the two following sets are well defined and not empty:
$$ M_+ = \left \{ \hat{x} \in \sta \, | \varphi(\hat{x}) = \underset{\hat{y}\in \sta}{\sup} \varphi(\hat{y})  \right \},$$
$$ M_- = \left \{ \hat{x} \in \sta \, | \varphi(\hat{x}) = \underset{\hat{y}\in \sta}{\inf} \varphi(\hat{y}) \right \}.$$
In order to show that $\varphi$ is constant, it suffices to prove that $M_+ \cap M_- \neq \emptyset$. Let $\hat{x} \in M_+$, by assumption we have $\varphi(\hat{x})=\int_{\mat} \varphi(v \cdot \hat{x}) q_s(\hat{x},v) d \mu(v)$. And it follows that:
\begin{equation*}
    0=\int_{\mat} [\varphi(\hat{x})-\varphi(v \cdot \hat{x})] q_s(\hat{x},v) d \mu(v).
\end{equation*}
Since $\hat{x} \in M_+$, for every $v \in \mat$ such that $v \cdot \hat{x}$ is well defined, $\varphi(\hat{x})-\varphi(v \cdot \hat{x}) \geq 0$ and by definition 
$q_s(\hat{x},v) \geq 0$ for every $v \in \mat$. It follows that for $\mu$-almost every $v$, $$[\varphi(\hat{x})-\varphi(v \cdot \hat{x})] q_s(\hat{x},v)=0.$$ 
If $\|vx\|=0$, we give the value $0$ to the quantity $[\varphi(\hat{x})-\varphi(v \cdot \hat{x})] q_s(\hat{x},v)$ and if $\|vx\| \neq 0$, then since $q_s(\hat{x},v)>0$ it implies that $\varphi(\hat{x})=\varphi(v \cdot \hat{x})$ and then $v \cdot \hat{x} \in M_+$. We now define the set 
$$C_+ = \linspan \left \{ \lambda x \, | \lambda \in \cc, \, \hat{x} \in M_+ \right \}.$$ For every $\hat{x} \in M_+$ and for $\mu$-almost every $v \in \mat$, $vx=0$ or $vx=\lambda y$ where $\hat{y} \in M_+$. It implies that $C_+$ is a $\supp \mu-$ invariant subspace not equal to $0$ since $M_+$ is not empty. The assumption {\bf (Irr)} implies that $C_+=\cc^d$. Now by {\bf (Cont)}, there exists a sequence $(v_n) \in T$ such that $\left (\frac{v_n}{\|v_n\|}\right)$ converges towards a rank-one matrix $v_\infty$. Since $C_+=\cc^d$, there exists $\hat{x}_+ \in M_+$ such that $v_\infty x_+ \neq 0$. Lemma \ref{lemma:vndot} implies that the sequence $(v_n \cdot \hat{x}+)$ converges towards $v_\infty \cdot \hat{x}_+$. By continuity of $\varphi$, the sequence $(\varphi(v_n \cdot \hat{x}_+))$ converges towards $\varphi(v_\infty \cdot \hat{x}_+)$. The invariance of $C_+$ by $\supp \mu$ implies that for every $n$, $v_n \cdot \hat{x}_+ \in M_+$, the sequence $(\varphi(v_n \cdot \hat{x}_+))$ is constant. We deduce that 
\begin{equation}\label{eq:appartientM}
    v_\infty \cdot \hat{x}_+ \in M_+.
\end{equation}
We denote by $z$ a unit vector of $\Im v_\infty$. Since $v_\infty$ is a rank one matrix, Equation \eqref{eq:appartientM} shows that $\hat{z} \in M_+$. 
By repeating and adapting the proof for $M_-$, we can show that $\hat{z} \in M_-$ proving that $M_+ \cap M_-$ is not empty and then $\varphi$ is constant. If $\varphi$ is a continuous complex-valued function invariant by $Q_s$, we write $\varphi=\Re (\varphi)+i \Im (\varphi)$ with $\Re (\varphi)$ and $\Im (\varphi)$ two continuous real-valued functions . Since $Q_s$ is a linear operator, 
\begin{equation}\label{eqcomplex}
    \Re (\varphi)+i \Im (\varphi)=\varphi=Q_s \varphi = Q_s \Re (\varphi)+ i Q_s \Im (\varphi).
\end{equation}
Now, let us remark that the set of real-valued functions is invariant by $Q_s$. By identifying the real and the imaginary part in Equation \eqref{eqcomplex} we obtain that $Q_s \Re (\varphi)=\Re (\varphi)$ and $Q_s \Im (\varphi)=\Im (\varphi)$. Since $Re (\varphi)$ and $\Im (\varphi)$ are two $Q_s$ real-valued continuous functions invariant by $Q_s$, it follows that $Re (\varphi)$ and $\Im (\varphi)$ are constant and we obtain that $\varphi$ is constant.

\end{proof}
We are now able to prove Proposition \ref{Prop:UniciteMesureInvariante}. 

\begin{proof}[Proof of Proposition \ref{Prop:UniciteMesureInvariante}]
Let $s \in I_\mu$.
We apply Lemma \ref{lemma:Raugi} for $X=\sta$ and $Q=Q_s$. Lemma \ref{Lemma:EstimationNormeHölder} ensures that for every $\varphi \in C(\sta)$, the sequence $(Q^n \varphi)$ is equicontinuous and Lemma \ref{lem:Constant} ensures that the only $Q_s$-invariant continuous functions are constant. Then, for every $s$, the operator $Q_s$ has a unique invariant probability measure.
\end{proof}

\begin{corollary}\label{Cor:Sigmaetes}
Assume {\bf (Irr)} and {\bf (Cont)} hold.
Let $s \in I_\mu$. There exist a unique probability measure $\sigma$ such that $\sigma \Gamma_s = k(s) \Gamma_s$ and a unique strictly positive continuous function $e_s : \sta \to \rr_+^*$ such that $\Gamma_s e_s = k(s) e_s$ and $\sigma(e_s)=1$. Moreover, $e_s$ is $\Bar{s}$-Hölder with $\Bar{s}= \min \{1,\frac{s}{2}\}$ and we have the relation $\eta^s=e_s \sigma$. So that, the support of $\eta^s$ is not included in a hyperplane.
\end{corollary}

\begin{proof}
The uniqueness of the invariant probability measure of $Q_s$ implies the uniqueness of the probability measure $\sigma$ such that $\sigma \Gamma_s = k(s) \Gamma_s$. Proposition \ref{Prop:Existencee_s} ensures the existence of a strictly positive continuous function $e_s : \sta \to \rr_+^*$ such that $\Gamma_s e_s = k(s) e_s$, by dividing by $\sigma(e_s)$, we may assume that $\sigma(e_s)=1$.
Now, let $\varphi : \sta \to \cc$ be a continuous function such that $\Gamma_s \varphi = k(s) \varphi$ and $\sigma(\varphi)=1$. By definition of $e_s$, we obtain:
\begin{equation}
    Q_s\left(\frac{\varphi}{e_s}\right)=\frac{\varphi}{e_s}.
\end{equation}
Since $\varphi$ is continuous and $e_s$ is continuous and strictly positive, it follows that $\frac{\varphi}{e_s}$ is continuous. Then, Lemma \ref{lem:Constant} implies that $\frac{\varphi}{e_s}$ is constant. It implies the existence of $\lambda \in \cc$ such that $\varphi=\lambda e_s$. Then, $1=\sigma(\varphi)=\sigma(\lambda e_s)=\lambda \sigma(e_s)=\lambda$. It follows that $\varphi=e_s$. The properties of $e_s$ have already been proved in Proposition \ref{Prop:Existencee_s}. Moreover we have the relation $\eta^s=e_s \sigma$. Indeed, for any continuous function $\varphi : \cc^d \to \cc$,
\begin{align*}
    (e_s \sigma)Q_s(\varphi) &= \int_{\sta} Q_s \varphi(\hat x) e_s(\hat x) d \sigma (\hat x) \\
    &= \int_{\sta} \frac{1}{k(s) e_s (\hat x)} \Gamma_s(\varphi e_s)(\hat x) e_s(\hat x) d \sigma (\hat x) \\
    &= \frac{1}{k(s)} \sigma \Gamma_s(\varphi e_s) \hspace*{4.2cm} \text{by definition of } \sigma \Gamma_s,\\
    &= \frac{1}{k(s)} k(s) \sigma(\varphi e_s) \hspace*{4cm}  \text{because } \sigma \Gamma_s = k(s) \sigma, \\
    &=e_s\sigma(\varphi)  \\
\end{align*}
Since we fix $\sigma(e_s)=1$, the measure $e_s \sigma$ is a probability measure on $\sta$. Then, by uniqueness of the invariant measure of $Q_s$, $e_s\sigma=\eta^s$.

\end{proof}

We can finally prove Theorem \ref{Constru}.

\begin{proof}
Let $s \in I_\mu$. Corollary \ref{Cor:Sigmaetes} ensures the existence and the uniquess of a probability measure $\sigma$ on $\sta$ such that $\sigma \Gamma_s = k(s) \Gamma_s$ and a unique continuous function $e_s : \sta \to \cc$ such that $\Gamma_s e_s = k(s) e_s$ and $\sigma(e_s)=1$. Moreover, $e_s$ is strictly positive and $\Bar{s}$-Hölder with $\Bar{s}= \min \{1,\frac{s}{2}\}$.
\end{proof}

\section{Quasi-compactness of the operators $Q_s$}\label{SecQuasiCompact}
In this section, we are interested in the spectral gap property of the operators $Q_s$. One way to establish the spectral gap property for an operator is to demonstrate its quasi-compactness. Therefore, we prove in this section the quasi-compactness of the operators $Q_s$.
We first recall the definition of the quasi-compactness.
\begin{definition}\label{DefQCompact}
Let $( \staalg , \| \cdot \|) $ a Banach space and  $Q$ a bounded operator on $\staalg$. The operator $Q$ is quasi-compact if there exists a decomposition of $\staalg$ into disjoint closed $Q$-invariant subspaces,
$$ \staalg = \mathcal{F} \oplus \mathcal{H}, $$
where $\mathcal F$ is a finite dimensional space, all the eigenvalues of $Q|_{\mathcal{F}}$ have modulus
$\rho(Q)$ and $\rho(Q|_{\mathcal{H}}) < \rho(Q)$ where $\rho$ denotes the spectral radius function.
\end{definition}
A way to prove that an operator is quasi-compact on a Banach space is Ionescu-Tulcea Marinescu Theorem \cite{ITM50}. We give here the version of \cite{HeHe01}.

\begin{theorem}[Theorem~2.5 in \cite{HeHe01}]\label{ITM1}
Let $( \staalg , \| \cdot \|) $ a Banach space and  $Q$ a bounded operator on $\staalg$.
Let $| \cdot |$ a continuous semi-norm on $\staalg$ and $Q$ a bounded operator on $( \staalg , \| \cdot \|) $ such that:
\begin{enumerate}
    \item \label{Hyp1} $Q( \{ f : f \in \staalg, \| f \| \leq 1 \} )$ is relatively compact in $(\staalg, | \cdot |)$.
    \item \label{Hyp2} There exists a constant $M$ such that $\text{for all } f \in \staalg,$
    $$ |Qf| \leq M |f|. $$
    \item \label{Hyp3} There exist $n \in \nn$ and real numbers $r,R$ such that $r < \rho(Q) \text{ and for all } f \in \staalg$,
    $$  \| Q^n f \| \leq R |f| + r^n \| f \|. $$
\end{enumerate}
Then $Q$ is quasi-compact.
\end{theorem}
We now state the main result of this section.
\begin{theorem}\label{QuasiCompacts}
Assume {\bf (Cont)} and {\bf (Irr)} hold. Then, for every $s \in I_\mu$, the operator $Q_s$ is quasi-compact on $(C^\alpha(\sta,\cc),|.|_\alpha)$.
\end{theorem}
To prove Theorem \ref{QuasiCompacts}, we verify the three assumptions of Theorem \ref{ITM1}. The key item to prove is Item $(3)$. In our context, we need to prove that there exist:
$0<r<\rho_\alpha(Q_s)$, $R \geq 0$, $n_0 \in \nn$ such that for every $f \in C^\alpha(\sta,\cc)$,
\begin{equation}
    \|Q_s^{n_0} \varphi\|_\alpha \leq R \|\varphi\|_\infty + r^{n_0}\|\varphi\|_\alpha.
\end{equation}
In order to prove this inequality, we use Lemma \ref{Lemma:EstimationNormeHölder}. For every $s \in I_\mu$, there exist $A_s$ and $B_s$ such that for every $\varphi \in C^\alpha(\sta,\cc)$: 
\begin{equation*}
    \|Q_s^n \varphi\|_\alpha \leq A_s \|\varphi\|_\infty + B_s h_\alpha(n) \|\varphi\|_\alpha,
\end{equation*}
We recall that $h_\alpha: \nn \to \rr_+$ is defined as follows:
\begin{equation*}
    h_\alpha(n)= \frac{1}{k(s)^n} \int_{\mat^n} \| \wedge^2 v_n...v_1 \|^\alpha \|v_n...v_1\|^{s-2\alpha} d \mu^{\otimes n}(v_1,...v_n).
\end{equation*}
We prove in Section \ref{sec:halpha} that there exists $n_0 \in \nn^*$ and $0<\lambda<1$ such that for every $n \geq n_0$:
\begin{equation}\label{eq:halphalambda}
    h_\alpha(n) \leq \lambda^n.
\end{equation}
In order to prove this result, we introduce the probability measure $\qq^s$ on $\out$ defined by
\begin{equation}\label{eq:definitionQQs}
    \qq^s=\int_{\sta} \qq^s_{\hat{x}} d \eta^s(\hat{x}).
\end{equation}
To obtain Equation \eqref{eq:halphalambda}, we first prove that:
\begin{equation}
    h_\alpha(n) \leq K_s \ee_{\qq^s}\left[\left \| \frac{\wedge^2W_n}{\|W_n\|^2} \right\|^\alpha\right],
\end{equation}
where $K_s \geq 0$ is a constant.  We define the process $(\widetilde{W}_n)$ such that for every $n \in \nn$, $ \Tilde{W_n}=\frac{W_n}{\|W_n\|}$. From Lemma 5.3 \cite{BL85}, we have
\begin{equation*}
    \|\wedge^2\widetilde{W}_n\|=a_1(\widetilde{W}_n)a_2(\widetilde{W}_n),
\end{equation*}
where $a_1(\widetilde{W}_n),a_2(\widetilde{W}_n)$ denotes the two largest singular values of $\widetilde{W}_n$ with multiplicity.
In Section \ref{Sec:Martingale}, we prove that the process $(a_2(\widetilde{W}_n))$ converges $\qq^s$-almost surely towards $0$ ensuring the convergence of $(h_\alpha(n))$ towards $0$. Then, by submultiplicativity of $(h_\alpha(n))$, we obtain Equation \eqref{eq:halphalambda}. Finally, by combining the results of Sections \ref{Sec:Martingale} and \ref{sec:halpha}, we prove the quasi-compactness of the operators $Q_s$ in Section \ref{sec:ProofQuasiCompact}.

\subsection{The $\qq^s$-almost sure convergence of $a_2(\widetilde{W}_n)$ towards 0}\label{Sec:Martingale}
The goal of this subsection is to prove the following proposition.
\begin{proposition}\label{Prop:Conva1a2}
The sequence $(a_2(\widetilde{W}_n))$ converges $\qq^s-$ almost surely towards $0$.
\end{proposition}
The proof is based on the study of the Radon-Nikodym derivative of $\qq_{\hat x}^s$ with respect to $\qq^s$ for any $\hat{x} \in \sta$.
In order to study it, we define the function $J_s : \mat \to \rr_+$ as follows: for every $A \in \mat$:
\begin{equation}\label{eq:DefIsA}
    J_s(A)=\int_{\sta} \frac{e_s(A \cdot \hat{y})}{e_s(\hat{y})}\|Ay\|^s d \eta^s(\hat{y}),
\end{equation}
where $\eta^s$ is the invariant probability measure of $Q^s$.
Since $e_s$ is bounded and strictly positive, $J_s$ is well defined and there exists $C_s \geq 0$ such that $J_s(A) \leq C_s \|A\|^s$ for every $A \in \mat$.
We first give some properties of this function and then study the martingale induced by the Radon-Nikodym derivative.
\subsubsection{Properties of the function $J_s$}
In this section, we give some properties of the function $J_s$.  
\begin{lemma}\label{lemma:ProprietesIs}
For every $s \in I_\mu$,
\begin{enumerate}
    \item The function $J_s: \mat \to \rr_+$ is continuous.
    \item There exists $K_s > 0$ such that for every $A \neq 0$, $J_s(A) \geq K_s \|A\|^s$.
    \item For every $A \neq 0$, $\frac{1}{\|A\|^s}J_s(A)=J_s(\frac{A}{\|A\|})$.
\end{enumerate}
\end{lemma}

\begin{proof}
To prove the continuity, we prove the sequential continuity.
Let $(B_n) \in \mat^{\nn}$ a sequence converging to $B_\infty$.
Firstly, for every $\hat{y} \in \sta$, by continuity of the map $e_s$ the sequence $(e_s(B_n \cdot \hat{y}) \|B_n y\|^s)$ converges towards $e_s(B_\infty \cdot \hat{y}) \|B_\infty y\|^s$.
Secondly, since $e_s$ is bounded and $(B_n)$ converges, the sequence $(e_s(B_n \cdot \hat{y}) \|B_n y\|^s)$ is bounded indepedently of $n$ and $y$. By Lebesgue dominated convergence, we obtain the first item. 
In order to prove the second item, we apply Lemma \ref{IneqSigma} because $\eta^s$ is not supported on a hyperplane (Corollary \ref{Cor:Sigmaetes}). There exists $c_s>0$ depending on $\eta^s$ such that for every $A \in \mat$,
$$ c_s\|A\|^s \leq \int_{\sta} \|Ax\|^s d \eta^s(\hat{x}). $$
Now, let us remark that since $e_s$ is a strictly positive bounded function, we have that for every $\hat x \in \sta$
$$ \|Ax\|^s = \frac{e_s(A \cdot \hat x)}{e_s(\hat x)} \|Ax\|^s \times \frac{e_s(\hat x)}{e_s(A \cdot \hat x)} \leq C_s \frac{e_s(A \cdot \hat x)}{e_s(\hat x)} \|Ax\|^s, $$
where $C_s=\underset{\hat x, \hat y \in \sta}{\sup}\frac{e_s(\hat x)}{e_s(\hat y)}>0$. By letting $K_s = \frac{c_s}{C_s}$, we obtain Item $2$.
The third item comes from immediate computation.

\end{proof}
We deduce the following corollary.
\begin{corollary}\label{Cor:Convergence}
Let $A \in \mat$, $(B_n) \in \mat^{\nn}$ a sequence converging to $B_\infty$, $\hat{x} \in \sta$ and $s \in I_\mu$. We assume that $\supp \sigma$ is not included in a hyperplane, then:
\begin{equation}\label{eq:ConvergenceContinuité}
    \frac{\|A B_n x \|^s e_s(B_n \cdot \hat{x})}{\int_{\sta} \|B_n y \| \frac{e_s(B_n \cdot \hat{y})}{e_s(\hat{y})} d \sigma(\hat{y})} \underset{n \to \infty}{\longrightarrow} \frac{\|A B_\infty x \|^s e_s(B_\infty \cdot \hat{x})}{\int_{\sta} \|B_\infty y \| \frac{e_s(B_\infty \cdot \hat{y})}{e_s(\hat{y})} d \sigma(\hat{y})} 
\end{equation}
\end{corollary}

\subsubsection{Radon-Nykodim derivative}

\begin{lemma}
For every $\hat{x} \in \sta$, the probability measure $\qq^s_{\hat{x}}$ is absolutely continuous with respect to $\qq^s$.  
\end{lemma}

\begin{proof}
Let $\hat x \in \sta$ and $n \in \nn^*$. We recall that $C= \underset{\hat z , \hat y \in \sta}{\sup}\frac{e_s(\hat z)}{e_s(\hat y)}< \infty$.
For every $A \in \mat$,
\begin{align*}
    q_n^s(\hat x,A) &= \frac{1}{k(s)^n} \frac{e_s(A \cdot \hat x)}{e_s(\hat x)}\|Ax\|^s \\
    &\leq \frac{C}{k(s)^n} \|A\|^s.
\end{align*}
By Corollary \ref{Cor:Sigmaetes}, the probability measure $\eta^s$ is not included in a hyperplane . Then, Lemma \ref{IneqSigma} implies the existence of $c_s>0$ such that:
$$ \|A\|^s \leq \frac{1}{c_s} \int_{\sta} \|Ay\|^s d \eta^s(\hat y).$$
It follows that:
$$ q_n^s(\hat x,A) \leq \frac{C^2}{c_s} \int_{\sta} q_n^s(\hat y,A) d\eta^s(\hat y).$$
Let $\varphi : \Omega \mapsto \rr_+$ depending only on the $n$ first coordinates.
\begin{align*}
    \qq_{\hat x}^s(\varphi) &= \int_{\mat} q_n^s(\hat x,W_n) d \mu^{\otimes n} \\
    &\leq \frac{C^2}{c_s} \int_{\mat} \int_{\sta} q_n^s(\hat y,W_n) d\eta^s(\hat y) d \mu^{\otimes n} \\
    &= \frac{C^2}{c_s} \qq^s(\varphi).
\end{align*}
Since $n$ and $\varphi$ were chosen arbitrarily, the conclusion necessarily follows.
\end{proof}
For every $\hat x$, we can then define the process $(P_n(\hat{x}))$:
\begin{equation}
    P_n(\hat{x})=\frac{d \qq^s_{\hat{x}}}{d \qq^s} \left |_{\mathcal{O}_n} \right. .
\end{equation}
We can check by computation that for every $n \in \nn$ and $\hat{x} \in \sta$:
\begin{equation}
    P_n(\hat{x}) = \frac{q^s_n(\hat{x},W_n)}{q_n^s(W_n)},
\end{equation}
where $q_n^s(W_n)= \int_{\sta} q^s_n(\hat{y},W_n) d \eta^s(\hat{y})$. By using the function $J_s$ defined in Equation \eqref{eq:DefIsA}, the process $P_n(\hat{x})$ can be written as follows:
\begin{equation*}
    P_n(\hat{x}) = \frac{e_s(W_n \cdot \hat{x}) \|W_n x\|^s}{e_s(\hat{x}) J_s(W_n)}.
\end{equation*}
By the properties of Radon-Nykodim derivatives, we obtain the following result.
\begin{lemma}\label{Martingale}
For every $\hat{x} \in \sta$, the process $(P_n(\hat{x}))$ is a bounded $(\mathcal{O}_n)$-martingale with respect to $\qq^s$ and then converges almost surely and in $L^1$ towards $\frac{d \qq^s_{\hat x}}{d \qq^s}$.
\end{lemma}

\begin{proof}
In order to apply the martingale convergence theorem, we only need to prove that for every $\hat x \in \sta$, the process $(P_n(\hat x))$ is bounded.
Let $\hat x \in \sta$. For every $n \in \nn^*$, $\qq^s$-almost surely, $W_n \neq 0$. Then Lemma \ref{lemma:ProprietesIs} implies the existence of $C_s>0$ such that $\frac{\|W_n\|^s}{J_s(W_n)} \leq C_s$. Since $\|W_n x\|^s \leq \|W_n\|^s$ and $e_s$ is bounded, it follows that $P_n(\hat{x})$ is bounded in $L^{\infty}(\qq^s)$. It follow that the martingale $P_n(\hat{x})$ is uniformly integrable and by martingale convergence theorem, we obtain the lemma.
\end{proof}

The following proposition is central to the proof of quasi-compactness and leverages the martingale property of the processes $(P_n(\hat{x}))$.

\begin{proposition}\label{Prop:ConvergenceWn}
Assume {\bf (Cont)} and {\bf (Irr)} hold. Then, for $\qq^s$- almost every $\omega \in \Omega$ and for every subsequence $(n_l(\omega))_l$ such that $\frac{W_{n_l}(\omega)}{\|W_{n_l}(\omega)\|}$ converges, the limit $W_\infty(\omega)$ has rank one.
\end{proposition}

\begin{remark}
Let us point that in Proposition \ref{Prop:ConvergenceWn}, the existence of the subsequences $(n_l(\omega))_l$ is not mentioned. However, by compacity, for $\qq^s$- almost every $\omega \in \Omega$, there exists at least one subsequence $(n_l(\omega))_l$ such that $\frac{W_{n_l}(\omega)}{\|W_{n_l}(\omega)\|}$ converges.
\end{remark}

\begin{proof}[Proof of Proposition \ref{Prop:ConvergenceWn}]

Let $\hat{x} \in \sta$, $p \in \nn^*$ and $n \in \nn$.
Let us recall that, according to Lemma \ref{Martingale}, the process $(P_n(\hat{x}))$ is a $\qq^s$-martingale. We define $\Delta_{n,p}$ as follows:
\begin{equation}
    \Delta_{n,p} = \sum_{l=0}^p \ee_{\qq^s}[P_{l+n+1}(\hat{x})^2 - P_{l}(\hat{x})^2 ].
\end{equation}
The martingale property implies that:
\begin{align*}
    \Delta_{n,p} &= \sum_{l=0}^n \ee_{\qq^s}[P_{l+p}(\hat{x})^2 - P_{l}(\hat{x})^2 ] \\
    &= \sum_{l=0}^n \ee_{\qq^s}[(P_{l+p}(\hat{x}) - P_{l}(\hat{x}))^2 ] \\
    &= \ee_{\qq^s} \left [\sum_{l=0}^n \ee_{\qq^s}[(P_{l+p}(\hat{x}) - P_{l}(\hat{x}))^2 | \mathcal{O}_l ] \right ].
\end{align*}
Since $(P_n(\hat{x}))$ is a bounded martingale, it follows that $\Delta_{n,p}$ converges when $n$ goes to infinity. It implies that:
\begin{equation*}
    \ee_{\qq^s} \left [\sum_{l=0}^\infty \ee_{\qq^s}[(P_{l+p}(\hat{x}) - P_{l}(\hat{x}))^2 | \mathcal{O}_l ] \right ] < \infty,
\end{equation*}
and
\begin{equation*}
    \underset{n \to \infty}{\lim} \ee_{\qq^s}[(P_{n+p}(\hat{x}) - P_{n}(\hat{x}))^2 | \mathcal{O}_n ] = 0.
\end{equation*}
Jensen inequality implies that:
\begin{equation*}
    \underset{n \to \infty}{\lim} \ee_{\qq^s}[|P_{n+p}(\hat{x}) - P_{n}(\hat{x})| | \mathcal{O}_n ] = 0.
\end{equation*}
We now explicit the quantity $ \ee_{\qq^s}[|P_{n+p}(\hat{x}) - P_{n}(\hat{x})| | \mathcal{O}_n ]$. By noticing that $ \frac{q^s(v_p...v_1W_n)}{q^s(W_n)}=\frac{J_s(v_p...v_1W_n)}{J_s(W_n)}$ we obtain that:
\begin{align*}
    &\ee_{\qq^s}[|P_{n+p}(\hat{x}) - P_{n}(\hat{x})| | \mathcal{O}_n ] \\
    &=\int_{\mat^p} \left | \frac{\|v_p...v_1W_n x\|^s e_s(v_p...v_1W_n \cdot \hat{x})}{J_s(v_p...v_1W_n) } - \frac{\|W_n x\|^s e_s(W_n \cdot \hat{x})}{J_s(W_n) } \right | 
     \frac{J_s(v_p...v_1W_n)}{J_s(W_n) e_s(\hat{x})} d \mu^{\otimes p}(v_1,...,v_p) \\
     &=\int_{\mat^p} \left | \|v_p...v_1W_n x\|^s e_s(v_p...v_1W_n \cdot \hat{x}) - \frac{\|W_n x\|^s e_s(W_n \cdot \hat{x})J_s(v_p...v_1W_n)}{J_s(W_n) } \right | 
     \frac{1}{J_s(W_n) e_s(\hat{x})} d \mu^{\otimes p}(v_1,...,v_p).
\end{align*}
Let us recall that $(\widetilde{W}_n)$ is the process defined by $\widetilde{W}_n = \frac{W_n}{\|W_n\|}$. Let $\omega \in \Omega$ and let $(n_l(\omega))_l$ a subsequence such that $(\widetilde{W}_{n_l}(\omega))_l$ converges to a matrix $W_\infty(\omega) \in \mat$. Lemma \ref{lemma:ProprietesIs} implies that
\begin{align*}
    \underset{l \to \infty}{\lim} J_s(W_{n_l}(\omega)) = J_s(W_\infty(\omega)).
\end{align*}
The probability measure $\eta^s$ is not included in a hyperplane \ref{Cor:Sigmaetes}. Since $W_\infty(\omega) \neq 0$, Lemma \ref{lemma:ProprietesIs} implies that $J_s(W_\infty(\omega))>0$. Let $\hat{x} \in \sta$ such that $W_\infty(\omega) x \neq 0$. By applying Corollary \ref{Cor:Convergence} and by taking the limit along the subsequence $(n_l(\omega))_l$, we obtain:
\begin{align*}
\int_{\mat^p} \left | \|v_p...v_1W_\infty(\omega) x\|^s e_s(v_p...v_1W_\infty(\omega)\cdot \hat{x}) - \frac{\|W_\infty(\omega)x\|^s e_s(W_\infty(\omega) \cdot \hat{x})J_s(v_p...v_1W_\infty(\omega))}{J_s(W_\infty(\omega)) } \right | \\
     \frac{1}{J_s(W_\infty(\omega)) e_s(\hat{x})} d \mu^{\otimes p}(v_1,...,v_p) = 0.
\end{align*} 
Since $\frac{1}{J_s(W_\infty) e_s(\hat{x})}>0$, it follows that for $\mu^{\otimes p}$-almost all $(v_1,...,v_p)$:
\begin{equation}\label{eq:AsAsp}
\|v_p...v_1W_\infty(\omega) x\|^s e_s(v_p...v_1W_\infty(\omega) \cdot \hat{x})J_s(W_\infty(\omega)) - \|W_\infty(\omega) x\|^s e_s(W_\infty(\omega) \cdot \hat{x})J_s(v_p...v_1W_\infty(\omega)) = 0.
\end{equation}
This equality also holds for $\hat{x} \in \sta$ such that $W_\infty(\omega) x =0$.
We now prove that $W_\infty(\omega)$ is a rank-one matrix.
Let us recall that the set $T$ contains all the finite products of matrices in $\supp \mu$ and the identity. By assumption {\bf (Cont)}, there exists a sequence $(v_k) \in T$ such that $\widetilde{v}_k=\frac{v_k}{\|v_k\|}$ converges towards a rank-one matrix $v_\infty$. There exists $k_0$ such that for every $k \geq k_0$, $v_k \neq 0$. Let $w \in T$, we can concatenate $v_k$ to $w$, by taking the limit along this subsequence we obtain that $\widetilde{v}_kw W_\infty$ converges towards $v_\infty w W_\infty$. Let us define $$L_\omega = \spn \{ \widetilde{w}y \, | \, \widetilde{w} \in T, \, y \in \Im(W_\infty(\omega)) \}.$$
The set $L_\omega \subset \cc^d$ is $\supp \mu$-invariant and not trivial. Indeed  $\Im(W_\infty(\omega)) \subset L_\omega$ and since $W_\infty(\omega) \neq 0$, $\Im(W_\infty(\omega)) \neq \{0\}$. By {\bf (Irr)}, we deduce that $L_\omega = \cc^d$ and the existence of $w$ such that $v_\infty w W_\infty(\omega) \neq 0$. If it was not the case, the set $L_\omega$ would be included in the kernel of $v_\infty$ which is not the null matrix and $L_\omega$ would not be $\cc^d$. Equation \eqref{eq:AsAsp} can be rewritten for $k$ large enough:
\begin{equation}\label{eq:AsAspomega}
\|\widetilde{v}_k wW_\infty(\omega) x\|^s e_s(\widetilde{v}_k wW_\infty(\omega)\cdot \hat{x})J_s(W_\infty(\omega)) - \|W_\infty(\omega) x\|^s e_s(W_\infty(\omega) \cdot \hat{x})J_s(\widetilde{v}_k wW_\infty(\omega)) = 0.
\end{equation}
It implies that:
\begin{equation}\label{Eq:Majoration}
    \|W_\infty(\omega) x\|^s \leq C_s \frac{1}{J_s(\widetilde{v}_k w(\omega)W_\infty(\omega))}\|\widetilde{v}_k w(\omega)W_\infty(\omega) x\|^s,
\end{equation}
where $C_s$ only depends on $e_s$ and $J_s(W_\infty(\omega))$. So if $v_\infty w(\omega) W_\infty(\omega) x = 0$, by taking the limit in the left side of \eqref{Eq:Majoration} when $k$ goes to infinity we also obtain that $W_\infty(\omega) x=0$ and conversely $W_\infty(\omega) x=0$ implies that $v_\infty w(\omega) W_\infty(\omega) x = 0$. We proved that 
\begin{equation}\label{Eq:EgaliteRang}
    \ker (v_\infty w(\omega) W_\infty(\omega)) = \ker (W_\infty(\omega)).
\end{equation}
By the properties of the rank, $$\rank(W_\infty(\omega)^*w(\omega)^*v_\infty^*v_\infty w(\omega) W_\infty(\omega)) \leq \min \{ \rank(v_\infty), \rank(w(\omega)), \rank (W_\infty(\omega)) \} \leq 1$$ since $\rank(v_\infty)=1$. This inequality is actually an equality since $v_\infty w(\omega) W_\infty(\omega)\neq 0$. Equation \eqref{Eq:EgaliteRang} then implies:
$$ \rank(W_\infty)=1.$$
 In conclusion, for $\qq^s-$almost all $\omega \in \Omega$ and for any subsequence $(n_l(\omega))_l$ such that $\frac{W_{n_l}(\omega)}{\|W_{n_l}(\omega)\|}$ converges, the limit $W_\infty(\omega)$ has rank one.
\end{proof}

We can finally prove Proposition \ref{Prop:Conva1a2}.

\begin{proof}[Proof of Proposition \ref{Prop:Conva1a2}]
For $\qq^s$-almost every $\omega$, the sequence $(\widetilde{W}_n(\omega))$ is well defined since for every $n \in \nn$, $\qq^s(W_n=0)=0$.  For such an $\omega$, we consider the sequence $(a_2(\widetilde{W}_n(\omega)))$. Since $\|\widetilde{W}_n(\omega)\|=1$, the sequence $(a_2(\widetilde{W}_n(\omega)))$ is bounded, we can extract a sequence $(n_l)_l$ such that $(a_2(\widetilde{W}_{n_l}(\omega)))$ converges towards a non-negative number $r(\omega) \geq 0$. By extracting a subsquence again, we can assume that $(\widetilde{W}_{n_l}(\omega))_l$ converges towards a matrix $W_\infty(\omega)$. By Proposition \ref{Prop:ConvergenceWn}, $\rank(W_\infty(\omega))=1$ and it follows that for $\qq^s$-almost every $\omega$, $a_2(W_\infty(\omega))=0$. By continuity of the singular values, it follows that $a_2(\widetilde{W}_{n_l}(\omega))$ converges to $0$. By uniqueness of the limit, we obtain that $r(\omega)=0$. Since $(a_2(\widetilde{W}_n(\omega)))$ is bounded and has a unique accumulation point, it converges to this limit, $\underset{n \to \infty}{\lim}a_2(\widetilde{W}_n(\omega))= 0$.

\end{proof}

\subsection{The properties of $h_\alpha$}\label{sec:halpha}
The goal of this section is to study the properties of the function $h_\alpha$ introduced in Equation \eqref{def:halphan}. We recall here its definition. For every $n \in \nn^*$:
\begin{equation}
    h_\alpha(n)= \frac{1}{k(s)^n} \int_{\mat^n} \| \wedge^2 v_n...v_1 \|^\alpha \|v_n...v_1\|^{s-2\alpha} d \mu^{\otimes n}(v_1,...v_n).
\end{equation}
The main result we want to prove is the following one.
\begin{proposition}\label{Prop:halphaConv}
The sequence $(h_\alpha(n))$ converges to $0$ when $n$ goes to infinity. In particular, there exists $0<\lambda<1$ and $n_0 \in \nn$ such that for every $n \geq n_0$:
\begin{equation}
    h_\alpha(n) \leq \lambda^n.
\end{equation}
\end{proposition}
Let us start by proving the submultiplicativity of the function $h_\alpha$.
\begin{lemma}\label{lemma:halphasubmult}
The function $h_\alpha$ is submultiplicative.
\end{lemma}
\begin{proof}
Let $n,m \in \nn$. We use the fact that for every $A,B \in \mat$, $\|\wedge^2 AB\| \leq \|\wedge^2 A\|\|\wedge^2 B\|$ and $\|AB\| \leq \|A\|\|B\|$, it follows that for every $(v_1,...v_n,v_{n+1},...,v_{n+m}) \in \mat^{n+m}$,
\begin{align*}
    \| \wedge^2 v_{n+m}...v_{n+1}v_n...v_1 \|^\alpha &\leq \|\wedge^2 v_{n+m}...v_{n+1}\|^\alpha \|\wedge^2v_n...v_1 \|^\alpha\\
    \|v_{n+m}...v_{n+1}v_n...v_1\|^{s-2\alpha} &\leq \|v_{n+m}...v_{n+1}\|^{s-2\alpha} \|v_n...v_1 \|^{s-2\alpha} 
\end{align*}
because the functions $t \mapsto t^\alpha$ and $t \mapsto t^{s-2\alpha}$ are increasing ($0<\alpha<\frac{s_{-}}{2}<\frac{s}{2}$). The submultiplicativity of the function $h_\alpha$ follows.
\end{proof}
We now prove Proposition \ref{Prop:halphaConv}. We recall that $\eta^s$ is the unique $Q_s$-invariant probability measure.

\begin{proof}[Proof of Proposition \ref{Prop:halphaConv}.]
Lemma \ref{lemma:ProprietesIs} shows the existence of a constant $K_s>0$ only depending on $e_s$ and $\eta^s$ such that for every $A \in \mat$, $\|A\|^s \leq K_s J_s(A)$. From this remark we deduce that for every $n \in \nn$
\begin{align*}
    h_\alpha(n) &= \int_{\mat^n} \left \| \frac{\wedge^2W_n}{\|W_n\|^2} \right\|^\alpha \frac{1}{k(s)^n} \|W_n\|^s d \mu^{\otimes n} \\
    &\leq K_s \int_{\mat^n} \left \| \frac{\wedge^2W_n}{\|W_n\|^2} \right\|^\alpha \frac{1}{k(s)^n} J_s(W_n) d \mu^{\otimes n}.
\end{align*}
We recall that $\frac{1}{k(s)^n} J_s(W_n)=q^s(W_n)$. Therefore, the previous inequality can be rewritten:
\begin{equation*}
    h_\alpha(n) \leq K_s \ee_{\qq^s}\left[\left \| \frac{\wedge^2W_n}{\|W_n\|^2} \right\|^\alpha\right].
\end{equation*}
Since $\left \|\frac{\wedge^2 W_n}{\|W_n\|^2} \right \|=a_1(\widetilde{W}_n)a_2(\widetilde{W}_n)$, we can apply Proposition \ref{Prop:Conva1a2} and since $a_1(\widetilde{W}_n)a_2(\widetilde{W}_n)$ is bounded by $1$, by Lebesgue dominated convergence theorem, it follows that $h_\alpha(n)$ converges to $0$ when $n$ goes to infinity. Now by Fekete Lemma, the submultiplicativity of $h_\alpha$ implies that $\frac{\log h_\alpha(n)}{n}$ converges towards $\underset{n \in \nn^*}{\inf} \frac{\log h_\alpha(n)}{n}<0$ since $h_\alpha(n)$ converges towards $0$. The existence of $n_0$ and $\lambda$ such that $h_\alpha(n) \leq \lambda^n$ for every $n \geq n_0$ follows.
\end{proof}

\subsection{Quasi-compactness of operators $Q_s$. Proof of Theorem \ref{QuasiCompacts}}\label{sec:ProofQuasiCompact}
This section is devoted to the proof of the quasi-compactness of the operators $Q_s$.

\begin{proof}[Proof of Theorem \ref{QuasiCompacts}]
To prove the quasi-compactness of the operators $Q_s$, it suffices to verify the three items of the Ionescu-Tulcea and Marinescu theorem (Theorem II.5 \cite{HeHe01}) and to prove that $\rho_\alpha(Q_s)=1$.

\begin{enumerate}
    \item From Lemma \ref{Lemma:EstimationNormeHölder}, we deduce the equicontinuity of the set $Q_s (\{ \varphi \in C^\alpha(\sta,\cc) \, | \, \|\varphi\|_\alpha \leq 1 \})$. Indeed, since $\|\varphi\|_\infty \leq \|\varphi\|_\alpha$, for every $\varphi \in C^\alpha(\sta,\cc)$:
    \begin{equation}\label{Q_sEquiContinu}
        \|Q_s \varphi\|_\alpha \leq \left [A_s + B_s h_\alpha(1) \right] \|\varphi\|_\alpha.
    \end{equation}
    In addition, since $Q_s$ is Markov, for every $\varphi \in C^\alpha(\sta,\cc)$ such that $\|\varphi\|_\alpha \leq 1$,
    $$ \|Q_s \varphi\|_\infty \leq \|\varphi\|_\infty \leq \|\varphi\|_\alpha \leq 1.$$
    It implies that for every $\hat{x} \in \sta$, the set $\{ Q_s \varphi(\hat{x}) \, | \, \varphi \in C^\alpha(\sta,\cc) \,  , \|\varphi\|_\alpha \leq 1 \}$ is bounded and then relatively compact in $\cc$. By Ascoli's Theorem, we deduce that $Q_s (\{ \varphi \in C^\alpha(\sta,\cc) \, | \, \|\varphi\|_\alpha \leq 1 \})$ is relatively compact in $(C^0(\sta,\cc),\|\cdot\|_\infty)$. Equation \eqref{Q_sEquiContinu} implies that for every $\varphi \in \{ \varphi \in C^\alpha(\sta,\cc) \, | \, \|\varphi\|_\alpha \leq 1\}$,
    $$ \|Q_s \varphi\|_\alpha \leq A_s + B_s h_\alpha(1).$$
    It follows that all the elements of $Q_s (\{ \varphi \in C^\alpha(\sta,\cc) \, | \, \|\varphi\|_\alpha \leq 1 \})$ are $\alpha$-Hölder, hence the relative compacity of $Q_s (\{ \varphi \in C^\alpha(\sta,\cc) \, | \, \|\varphi\|_\alpha \leq 1 \})$ in $(C^\alpha(\sta,\cc),\|\cdot\|_\infty)$ and we obtain Item $(1)$.
    \item $Q_s$ is a Markov Operator. Then, for every $f \in C^\alpha(\sta,\cc)$,
    $$ \|Q_sf\|_\infty \leq \|f\|_\infty.$$
    \item Lemma \ref{Lemma:EstimationNormeHölder} implies that for every $n \in \nn$ and for every $\varphi \in C^\alpha(\sta,\cc)$:
    \begin{equation}\label{eq:Item3QuasiCompact}
        \|Q_s^n \varphi\|_\alpha \leq A_s \|\varphi\|_\infty + B_s h_\alpha(n) \|\varphi\|_\alpha.
    \end{equation}
    By Proposition \ref{Prop:halphaConv}, there exist $n_0$ and $0< \lambda <1 $ such that for every $n \geq n_0$ $B_s h_\alpha(n) \leq \lambda^{n}$, for this $n_0$, Equation \eqref{eq:Item3QuasiCompact} implies that:
    \begin{equation} \label{eq:Item3Lambda}
        \|Q_s^{n_0} \varphi\|_\alpha \leq A_s \|\varphi\|_\infty + \lambda^{n_0} \|\varphi\|_\alpha.
    \end{equation}
    It remains to prove that $\rho_\alpha(Q_s)=1$. From the equality $Q_s 1=1$, we deduce that:
    $$\rho_\alpha(Q_s) \geq 1.$$
    Proposition \ref{Prop:halphaConv} implies that $h_\alpha$ is bounded by a constant $C$. For every $\varphi \in C^\alpha(\sta,\cc)$, $\|\varphi\|_\infty \leq \|\varphi\|_\alpha$, therefore Equation \eqref{eq:Item3QuasiCompact} implies:
    $$ \|Q_s^n\|^{1/n}_\alpha \leq \left [A_s + B_s C \right]^{1/n}. $$
    By taking the limit when $n$ goes to infinity, we obtain:
    $$ \rho_\alpha(Q_s) \leq 1.$$
    We proved that $\rho_\alpha(Q_s)=1$. Equation \eqref{eq:Item3Lambda} proves Item $(3)$.
\end{enumerate}
In conclusion, $Q_s$ is a quasi-compact operator on $(C^\alpha(\sta,\cc), \|.\|_\alpha)$ by applying Theorem II.5 \cite{HeHe01}.
\end{proof}

\section{The analyticity of $s \mapsto k(s)$.}\label{Sec:Analyticks}
This section is devoted to the proof of Theorem \ref{Th:Analyticks}.

\subsection{The analyticity of $z \mapsto \Gamma_z$}
Let us first recall the definition of the operators $\Gamma_z$. If $z=s+it$ where $s \in I_\mu$, we define the operator $\Gamma_z$ as follows: for $f \in C^\alpha(\sta,\cc)$ and $\hat{x} \in \sta$
$$\Gamma_z f(\hat{x}) = \int_{\mat} f( v \cdot \hat{x}) e^{z \log \|vx\|} d \mu(v). $$
When $\|vx\|=0$, since $s>s_{-}$,we extend by continuity by setting $e^{z \log \|vx\|}=0$.
The goal of this section is to prove the following theorem.
\begin{theorem}\label{AnalyticLog}
    Assume {\bf (Cont)} and {\bf (Irr)} hold. Let $\alpha= \min\{\frac{s_{-}}{3},1\}$. Then,
    as an endomorphism of $C^\alpha(\sta,\cc)$,
    $$\begin{array}{ccccc}
              z & \longmapsto & \Gamma_z
    \end{array}$$
    is analytic on the strip $\{z\in\mathbb C:\Re(z)\in I_\mu\}$.
\end{theorem}
We first prove that for every $\{z\in\mathbb C:\Re(z)\in I_\mu\}$, and $f \in C^\alpha(\sta,\cc)$, $\Gamma_z f$ is indeed in $C^\alpha(\sta,\cc)$. It is based on Lemma 5.2 in \cite{BHP24}.
\begin{lemma}[Lemma 5.2~\cite{BHP24}]\label{lem:Inegalitef_A}
    Let $z\in \cc$ be such that $\Re(z)>0$. Fix $\alpha\in (0,\min\{1,\Re(z/2)\})$ and $f\in C^\alpha(\sta,\cc)$. Then for any matrix $A\in \mat$, the function $f_A:\hat x\mapsto e^{z\log \|Ax\|}f(A\cdot \hat x)$ is $\alpha$-Hölder continuous with
    $$\|f_A\|_\alpha\leq 2^{\alpha^*}\frac{|z/2|}{\alpha^*}\|A\|^{\Re(z)}\|f\|_\alpha$$
    with $\alpha^*=\min\{1,\Re(z/2)\}$.
\end{lemma}
Now let us explain the method of proof of Theorem \ref{AnalyticLog}. We first show that for every $z \in \cc$ such that $\Re(z) \in I_\mu$ and $n \in \nn$ the operator $\Gamma_n^z$ defined below is well defined and bounded. For every $f \in C^\alpha(\sta,\cc)$ and $\hat{x} \in \sta$,
\begin{equation}
    \Gamma_n^z f (\hat x) = \int_{\mat} \log^n \|vx\| f(v \cdot \hat{x}) \|vx\|^z d \mu(v).
\end{equation}
Second, we prove that for every $z \in \cc$ such that $\Re(z) \in I_\mu$, there exists $\theta_z>0$ such that for every $w \in \cc$ such that $|w| \leq \theta_z$,
\begin{equation}\label{eq:developpementserie}
    \Gamma_{z+w}=\sum_{n=0}^{\infty} \frac{w^n}{n!} \Gamma_n^z.
\end{equation}
In order to prove this equality, we first prove that that for every $0<\theta<\theta_z$, the series $\sum_{n=0}^{\infty} \frac{\theta^n}{n!} \|\Gamma_n^z\|_\alpha$ is convergent and by Fubini theorem we deduce Equation \eqref{eq:developpementserie}.
Proving this convergence requires to study the functions $(F_{n,z})$ defined on $\rr_+^*$ by:
\begin{equation*}
    F_{n,z}(t)= \log^n(t) e^{z \log(t)},
\end{equation*}
for every $n \in \nn$ and $z$ such that $\Re(z)>0$. Indeed, our goal will be to study the Hölder property of the functions $(G_{n,z,f,A})_{n \geq 1, Re(z) \in I_\mu}$ for $A \in \mat$ defined by the relation:
$$G_{n,z,f,A}( \hat x)=\log^n \|Ax\| f(v \cdot \hat{x}) \|Ax\|^z.$$
These functions have been studied in Appendix A of \cite{BHP24}. We recall here a useful lemma. For a $\alpha$-Hölder continuous function $F:[0,t]\to \cc$, let
$$m_{t,\alpha}(F)=\sup_{u,s\in [0,t]:u\neq s}\frac{|F(u)-F(s)|}{|u-s|^\alpha}$$
and
$$\|F\|_t=\sup_{s\in [0,t]}|F(s)|.$$
\begin{lemma}[Lemma~B.4 \cite{BHP24}]\label{lem:B4}
    \label{lem:Holder norm F_nz}
    For any $z\in\mathbb C$ such that $\Re(z) > 0$, $r\in (0,1]$ such that $\Re(z)>r$ and $\theta>0$, there exists $C>0$ such that 
    for any $n\geq 1$, $t> 0$,
    $$\|F_{n,z}\|_t\leq e^{-n}\left(\frac{n}{\gamma_\infty}\right)^n \max\left(1,t^{\Re(z)+\theta}\right)$$
    with $\gamma_\infty=\min(\Re(z),\theta)$ and
    $$m_{t,r}(F_{n,z})\leq C e^{-n}\left(\frac{n}{\gamma_0}\right)^n \max\left(1,t^{\Re(z)-r+\theta}\right)$$
    with $\gamma_0=\min(\Re(z)-r,\theta)$.
\end{lemma}
We now are able to prove that the functions $G_{n,z,f,A}$ are $\alpha$-Hölder.

\begin{lemma}\label{lem:EstimationGnzfA}
Let $n \geq 1$, $z \in \cc$ such that $\Re(z) \in I_\mu$, $A \in \mat$, $\theta>0$ and $f : \sta \to \cc$ be an $\alpha$-Hölder function. Then, the function $G_{n,z,f,A}: \sta \to \cc$ is $\alpha$-Hölder and we have the following bound:
\begin{equation}
    \|G_{n,z,f,A}\|_\alpha \leq C e^{-n}\left(\frac{n}{\gamma}\right)^n \max\left(\|A\|^{s_-},\|A\|^{\Re(z)+\theta}\right)\|f\|_\alpha,
\end{equation}
with $\gamma=\min(\Re(z)-s_-,\theta)$.
\end{lemma}

\begin{proof}
Let us recall that $\alpha=\min \{ 1, \frac{s_{-}}{3} \}$. Now let $n \geq 1$ and $z \in \cc$ such that $\Re(z) \in I_\mu$. Our strategy is to split the exponential $e^{z \log \|Ax\|}$ in order to use Lemma \ref{lem:Inegalitef_A} and Lemma \ref{lem:B4}. Fix $r=s_- - \alpha$ We write $G_{n,z,f,A}$ as follows, for every $\hat x \in \sta$:
$$G_{n,z,f,A}(\hat x)=\log^n{\|Ax\|}e^{(z-s_{-})\log \|Ax\|} \|Ax\|^{\alpha} e^{r \log \|Ax\|} f(A \cdot \hat x).$$
In the rest of the proof we denote by $H$ the function defined by $H(\hat x)=e^{r\log \|Ax\|} f(A \cdot \hat x).$ and $L_n$ the function defined by $L_n(\hat x)=\log^n \|Ax\| e^{(z-s_{-})\log \|Ax\|}$. It follows that
$$ G_{n,z,f,A}(\hat x)= \|Ax\|^{\alpha}L_n(\hat x)H(\hat x).$$
Let us start by estimating the supremum norm of $G_{n,z,f,A}$. We estimate the supremum norm of $H$ as follows:
$$ \|H\|_\infty \leq \|A\|^{s_- - \alpha} \|f\|_\infty.$$
By definition, $L_n(\hat x)=F_{n,z-s_-}(\|Ax\|)$. Then, Lemma~\ref{lem:B4} applied to $t=\|A\|$ and $z-s_-$ gives:
\begin{equation}\label{eq!EstiméeGnz}
    \|G_{n,z,f,A}\|_\infty \leq e^{-n}\left(\frac{n}{\gamma}\right)^n \max\left(1,\|A\|^{\Re(z)-s_-+\theta}\right) \|A\|^{s_-} \|f\|_\infty,
\end{equation}
with $\gamma= \min \{ \Re(z) -s_-, \theta \}$. We now estimate the $\alpha$-Hölder coefficient of $G_{n,z,f,A}$. Let $\hat x, \hat y \in \sta$ such that $\hat x \neq \hat y$. We obtain:
\begin{align}
    G_{n,z,f,A}(\hat x) - G_{n,z,f,A}(\hat y) &= [H(\hat x)-H(\hat y)] \|Ax\|^{\alpha} L_n(\hat x) \label{eq:PartieduH} \\
    &+ H(\hat y) [\|Ax\|^\alpha L_n(\hat x)-\|Ay\|^\alpha L_n(\hat y)] \label{eq:PartieduLN}.
\end{align}
We first estimate the right term of Equation \eqref{eq:PartieduH}. To apply Lemma \ref{lem:Inegalitef_A}, it suffices to verify that $r \geq 2 \alpha$. By definition of $r$ and $\alpha$, the inequality can be rewritten $s_- \geq 3 \alpha$. Since $\alpha \leq \frac{s_-}{3}$, the inequality is verified. Then, Lemma \ref{lem:Inegalitef_A} applied to $r$ and Lemma \ref{lem:B4} applied to $\|A\|$ and $z-s_-$ give the existence of a constant $C$ only depending on $r$ such that:
\begin{equation}\label{Eq:Estimee1erTerme}
    | [H(\hat x)-H(\hat y)] \|Ax\|^{\alpha} L_n(\hat x) | \leq C e^{-n}\left(\frac{n}{\gamma}\right)^n \max\left(1,\|A\|^{\Re(z)-s_-+\theta}\right) \|A\|^{s_-} \|f\|_\alpha \d (\hat x, \hat y)^\alpha,
\end{equation}
with $\gamma= \min \{ \Re(z) -s_-, \theta \}$. Second, we estimate the term in Equation \eqref{eq:PartieduLN}. Let us remark that $\|Ax\|^\alpha L_n(\hat x)=F_{n,z-r}(\|Ax\|)$ for every $\hat x \in \sta$. Now, let us verify that $\alpha < \Re(z-r)$ in order to apply the second part of Lemma \ref{lem:B4}.
By definition of $r$, $\Re(z)-r = \Re(z)-s_-+ \alpha$ and since $\Re(z)>s_-$ it follows that $\Re(z)-r>\alpha$. 

\begin{align*}
    |\|Ax\|^\alpha L_n(\hat x)-\|Ay\|^\alpha L_n(\hat y)| &= |F_{n,z-r}(\|Ax\|)-F_{n,z-r}(\|Ay\|)| \\
    &\leq C e^{-n}\left(\frac{n}{\gamma}\right)^n \max\left(1,\|A\|^{\Re(z)-s_-+\theta}\right) | \|Ax\|-\|Ay\| |^\alpha \\
    &\leq C e^{-n}\left(\frac{n}{\gamma}\right)^n \max\left(1,\|A\|^{\Re(z)-s_-+\theta}\right) \|A\|^\alpha d(\hat x, \hat y)^\alpha,
\end{align*}
with $\gamma=\min(\Re(z)-s_-,\theta)$. The last inequality comes from Lemma 5.4 \cite{BHP24}.
Now by applying Lemma \ref{lem:Inegalitef_A} to $H$ we bound the term in Equation \eqref{eq:PartieduLN} as follows:
\begin{equation}\label{eq:estimee2emeterme}
    |H(\hat y) [\|Ax\|^\alpha L_n(\hat x)-\|Ay\|^\alpha L_n(\hat y)]| \leq C e^{-n}\left(\frac{n}{\gamma}\right)^n \max\left(1,\|A\|^{\Re(z)-s_-+\theta}\right) \|A\|^{s_-} \|f\|_\alpha d(\hat x, \hat y)^\alpha.
\end{equation}
By combining Equation \eqref{Eq:Estimee1erTerme} and Equation \eqref{eq:estimee2emeterme}, we obtain that there exists $C>0$ such that:
\begin{equation}\label{eq:EstimeeGnzalpha}
    m_\alpha(G_{n,z,f,A}) \leq C e^{-n}\left(\frac{n}{\gamma}\right)^n \max\left(1,\|A\|^{\Re(z)-s_-+\theta}\right) \|A\|^{s_-} \|f\|_\alpha.
\end{equation}
Finally, by combining Equation \eqref{eq:EstimeeGnzalpha} and Equation \eqref{eq!EstiméeGnz}, it follows that there exists $C>0$ such that:
\begin{equation}
    \|G_{n,z,f,A}\|_\alpha \leq C e^{-n}\left(\frac{n}{\gamma}\right)^n \max\left(\|A\|^{s_-},\|A\|^{\Re(z)+\theta}\right)\|f\|_\alpha,
\end{equation}
with $\gamma=\min(\Re(z)-s_-,\theta)$.
\end{proof}

We are now in position to prove Theorem \ref{AnalyticLog}.
\begin{proof}[Proof of Theorem \ref{AnalyticLog}]
Let $z \in \cc$ such that $\Re(z) \in I_\mu$. Our goal is to prove the convergence of the series of $(\frac{|w|^n}{n!}\|\Gamma_n^z\|_\alpha)$. Let $f : \sta \to \cc$ be a $\alpha$-Hölder continuous function. 
$$ \Gamma_0^z f( \hat x) = \int_{\mat} e^{z \log \|vx\|} f( v \cdot x) d \mu(v).$$
Lemma \ref{lem:Inegalitef_A} applied to $\hat x \mapsto e^{z \log \|vx\|} f( v \cdot x)$ and the triangular inequality give the existence of $C>0$ such that:
\begin{equation}
    \|\Gamma_0^z f\|_\alpha \leq C \int_{\mat} \|v\|^{\Re(z)} d \mu(v) \|f\|_\alpha.
\end{equation}
It follows that:
\begin{equation}
    \|\Gamma_0^z\|_\alpha \leq C \int_{\mat} \|v\|^{\Re(z)} d \mu(v).
\end{equation}
Now, for every $n \geq 1$ and $\hat x \in \sta$
$$\Gamma_n^z f ( \hat x ) = \int_{\mat} G_{n,z,f,v}(\hat x) d \mu(v).$$
Lemma \ref{lem:EstimationGnzfA} and the triangular inequality imply that:
\begin{equation}
    \|\Gamma_n^z f\|_\alpha  \leq C e^{-n}\left(\frac{n}{\gamma}\right)^n \int_{\mat} \max\left(\|v\|^{s_-},\|v\|^{\Re(z)+\theta}\right) d\mu(v)\|f\|_\alpha,
\end{equation}
with $\theta >0$ and $\gamma= \min \{ \Re(z) -s_-, \theta \}$. And then,
\begin{equation}
    \|\Gamma_n^z \|_\alpha  \leq C e^{-n}\left(\frac{n}{\gamma}\right)^n \int_{\mat} \max\left(\|v\|^{s_-},\|v\|^{\Re(z)+\theta}\right) d\mu(v).
\end{equation}
We fix now $\theta=s_+-\Re(z)>0$ since $\Re(z)<s_+$. For this specific $\theta$,
\begin{align*}
    \sum_{n=0}^\infty \frac{|w|^n}{n!} \|\Gamma_n^z\|_\alpha &\leq C \int_{\mat} \|v\|^{\Re(z)} d \mu(v) 
    + C \sum_{n=1}^\infty \frac{|w|^n}{n!} e^{-n}\left(\frac{n}{\gamma}\right)^n \int_{\mat} \max\left(\|v\|^{s_-},\|v\|^{s_+}\right) d\mu(v).
\end{align*}
By developing the exponential, one can prove that $\underset{n \in \nn}{\sup} \frac{e^{-n}n^n}{n!} \leq 1$. It follows that:
\begin{align*}
    \sum_{n=0}^\infty \frac{|w|^n}{n!} \|\Gamma_n^z\|_\alpha &\leq C \int_{\mat} \|v\|^{\Re(z)} d \mu(v) 
    + C \sum_{n=1}^\infty \left (\frac{|w|}{\gamma}\right)^n \int_{\mat} \max\left(\|v\|^{s_-},\|v\|^{s_+}\right) d\mu(v).
\end{align*}
The integrals converge by assumption, and for $0 \leq |w| < \gamma$, the series converges. The analyticity of $z \mapsto \Gamma_z$ follows.
\end{proof}

\subsection{Proof of Theorem \ref{Th:Analyticks}}
We can finally prove the main theorem.
\begin{proof}[Proof of Theorem \ref{Th:Analyticks}]
Let $s \in I_\mu$ and $\alpha= \min \{ 1, \frac{s_{-}}{3} \}$. According to Theorem \ref{Constru}, $k(s)$ is a simple eigenvalue of $\Gamma_s$. 
By Theorem \ref{AnalyticLog}, $z \to \Gamma_z$ is analytic on $\{ z \in \cc \, | \, \Re(z) \in I_\mu \}$. By applying Theorem VII.8 of \cite{kato2013perturbation} there exists an analytic continuation of the function $k$ on $\{ z \in \cc \, | \, \Re(z) \in I_\mu \}$. We then deduce the real analyticity the map $s \mapsto k(s)$ on $I_\mu$. Then, as mentioned in Remark \ref{Remark:Positivityk}, for every $s \in I_\mu$, $k(s)>0$. It follows that $\mathbf{P}$ is analytic on $I_\mu$.
\end{proof}

\section{Dynamical Systems Perspective}

The goal of this section is to link the previous results with the thermodynamics formalism (see for example \cite{bowen2006ergodic}).
In particular, the probability measure $\qq^s$ can be seen as a measure satisfying a variational principle. We link the previous sections with the results of Morris \cite{morris2018ergodic} or the results of Feng and Käenmäki \cite{Fe11}. We start by introducing new objects for this section.
Let $\mathcal{A}$ be a Polish space endowed with a metric $d_{\mathcal{A}}$.
We consider the space of infinite sequences $\out:=\mathcal{A}^{\nn}$ and we note $a=(a_1,...,a_n,...) \in \mathcal{A}^\nn$. Let $\mathcal M$ be the Borel $\sigma$-algebra on $\mathcal{A}$. For $n\in\nn$, let $\outalg_n$ be the  $\sigma$-algebra on $\Omega$ generated by the $n$-cylinder sets, i.e.\ $\outalg_n = \pi_n^{-1}(\mathcal M^{\otimes n})$. We equip the space $\Omega$ with the smallest $\sigma$-algebra $\outalg$ containing $\outalg_n$ for all $n\in \nn$. Let $\mu$ be a measure on $\mathcal A$. We identify $O_n\in \mathcal M^{\otimes n}$ with $\pi_n^{-1}(O_n)$, a function $f$ on $\mathcal M^{\otimes n}$  with $f \circ \pi_n$ and a measure $\mu^{\otimes n}$ with the measure $\mu^{\otimes n} \circ \pi_n$.
On $\Omega$, we define the distance $d_{\Omega}$: for $a,b \in \Omega$, 
\begin{equation}
    d_{\Omega}(a,b)=\sum_{k=1}^{+ \infty}2^{-k} \min \{d_A(a_k,b_k),1\}.
\end{equation}
On $\mathcal{A}^\nn$, we define the left-shift $\theta$:
$$\begin{array}{lrcl}
\theta : & \Omega & \longrightarrow & \Omega \\
    & (a_1,a_2,...) & \longmapsto & (a_2,a_3,...) \end{array}.$$
We denote by $\mathcal P(\Omega)$ the set of probability measures on $\Omega$. For $\pp \in \mathcal P(\Omega)$, we denote by $\pp_n$ its marginal on $\mathcal A^n$.
The set of probability measures on $\Omega$ invariant by $\theta$ will be denoted $\mathcal P_\theta$. The set of real-valued bounded Lipchitz functions on $\Omega$ will be denoted $\mathrm{Lip_b}(\Omega,\rr)$. If $f \in \mathrm{Lip_b}(\Omega,\rr)$, its Lipchitz norm is denoted $\|f\|_1$ which is the $\alpha$-Hölder for $\alpha=1$
Let $d \in \nn$, we define a measurable application $v: \mathcal{A} \mapsto \mat$: for every $a \in \mathcal{A}$,
\begin{equation}
    v(a)=v_a.
\end{equation}
We can then define the pushforward measure of $\mu$: $\bar{\mu}=\mu \circ v^{-1}$. Abusing the notation, we still denote this pushforward measure by $\mu$ in the following. 
We now assume that $\mu$ verifies the assumptions {\bf (Irr)} and {\bf (Cont)}. We still denote by $I_\mu$ the set of $s \in \rr_+$ such that $\int_{\mathcal A} \|v_a\|^s d \mu(a) < \infty$.
We can then define again the operators $\Gamma_s$: for every $f \in C^\alpha(\sta,\cc)$ and $\hat x \in \sta$:
$$ \Gamma_s f (\hat x)=\int_{\sta} f(v_a \cdot \hat x) \|v_ax\|^s d \mu(a).$$
In this setting, for every $s \in I_\mu$:
$$ k(s)=\underset{n \to \infty}{\lim} \left(\int_{\mathcal A^n} \|v_{a_n}...v_{a_1}\|^s d \mu^{\otimes n}(a_1,...,a_n)\right)^{1/n}.$$
By repeating the proofs for the pushforward measure, all the results in the previous sections also hold for the measure $\mu$. We state here a summary of the results.
\begin{proposition}
For every $s \in I_\mu$, there exist a unique strictly positive function $e_s: \sta \to \rr_+^*$ and a unique probability measure $\sigma$ over $\sta$ such that:
\begin{itemize}
    \item $\Gamma_s=k(s)e_s$.
    \item $\sigma \Gamma_s=k(s) \sigma$.
    \item $\sigma(e_s)=1.$
\end{itemize}
For every $\hat x \in \sta$, there exists a probability measure $\qq_x^s$ on $\mathcal A^{\nn}$ such that for any $n \in \nn$ and cylinder $O_n \in \mathcal O_n$,
$$ \qq_x^s(O_n)=\int_{O_n} q^s_n(\hat x, a_1,...,a_n) d \mu^{\otimes n}(a_1,...,a_n),$$
where $q^s_n(\hat x, a_1,...,a_n)= \frac{1}{k(s)^n} \frac{e_s(v_{a_n}...v_{a_1} \cdot \hat x)}{e_s(\hat x)}\|v_{a_n}...v_{a_1}x\|^s $.
Moreover, the operators $Q_s$ defined in Equation \eqref{Def:Qs} have a unique invariant probability measure $\eta^s$ over $\sta$.
\end{proposition}
As in the previous sections, we denote by $\qq^s$ the probability measure over $\mathcal{A}$ defined by:
\begin{equation}
    \qq^s= \int_{\sta} \qq_x^s d \eta^s(\hat{x}).
\end{equation}
In order to prove a variational principle, we prove that the probability measure $\qq^s$ is ergodic for the shift operator. For this purpose, we show an exponential decay of correlations. The proof is detailed in Section \ref{sec:Decay}. By using this result, we are able to prove the variational principle in Section \ref{Sec:VariationalPrinciple}.

\subsection{Exponential Decay of Correlations}\label{sec:Decay}

The goal of this subsection is to prove an exponential decay of correlation for the probability measure $\qq^s$. The spectral gap property of the operators $Q_s$ is crucial. Before stating the main result of this section, let us give a nice decomposition of the quasi-compact operators $Q_s$.
For every $s \in I_\mu$, there exists $m \in \nn^*$ such that $Q_s$ can be written as,
\begin{equation}\label{eq:decompoquasicompact}Q_s=\quad\sum_{\ell=0}^{m-1} e^{i\frac{\ell}{m}2\pi}f^s_\ell\nu^s_\ell \quad +\quad T_s,
\end{equation}
in the sense that for all bounded and measurable functions $f$, we have
$$Q_s f=\quad\sum_{\ell=0}^{m-1} e^{i\frac{\ell}{m}2\pi}f^s_\ell \nu^s_\ell(f) \quad +\quad T_s(f).$$
Here, we have that 
$$f^s_\ell\in C^\alpha(\sta,\cc),\quad \nu^s_\ell\in C^\alpha(\sta,\cc)^*\quad \mbox{and}\quad \nu_\ell(f^s_j)=\delta_{\ell,j}$$
and 
$$T_s:C^\alpha(\sta,\cc)\to C^{\alpha}(\sta,\cc),\quad \rho_\alpha(T_s)<1\quad \mbox{and} \quad T_sf_\ell=0,\quad \nu^s_\ell T_s=0$$
for any $\ell,j\in \{0,\dotsc,m-1\}$. Moreover, Lemma \ref{lem:Constant} implies that $f_0^s=\mathbb{1}$, where $\mathbb{1}$ is the constant function equal to $1$. In the following, we denote by $z_m$ the number $e^{\frac{i}{m} 2 \pi}$.
We now state the main proposition of this section.

\begin{proposition}\label{prop:decay}
Let $s \in I_\mu$. There exists $C>0$, $0<\lambda<1$ such that for every $n \in \nn^*$ and $f,g \in \mathrm{Lip}(\Omega,\rr)$,
\begin{equation}
    \left | \frac{1}{m} \sum_{j=0}^{m-1} \ee_{\qq^s}[f g \circ \theta^{nm+j}]- \ee_{\qq^s}[f] \ee_{\qq^s}[g] \right | \leq C \|f\|_1 \|g\|_1 \lambda^n.
\end{equation}
\end{proposition}
The proof of this proposition is decomposed in several lemmas. In Lemma \ref{lem:SommePeriode} we prove an identity on the sum over $m$. Lemma \ref{lem:Inegalitefpmgpm} gives the result for an approximation of $f$ and $g$. Finally we prove Proposition \ref{prop:decay} by choosing wisely the approximations of $f$ and $g$.

\begin{lemma}\label{lem:SommePeriode}
Let $s \in I_\mu$. For every $f \in C^\alpha(\sta,\cc)$ and $n \in \nn^*$,
\begin{equation}
    \frac{1}{m} \sum_{j=0}^{m-1} Q_s^{nm+j} f = \eta^s(f)+ \frac{1}{m} \sum_{j=0}^{m-1} T_s^{nm+j} (f).
\end{equation}
\end{lemma}

\begin{proof}
Let $s \in I_\mu$, $f \in C^\alpha(\sta,\cc)$ and $n \in \nn^*$.
Equation \eqref{eq:decompoquasicompact} yields:
\begin{align*}
     \frac{1}{m} \sum_{j=0}^{m-1} Q_s^{nm+j} f &=  \frac{1}{m} \sum_{j=0}^{m-1} \sum_{l=0}^{m-1} z_m^{l(nm+j)}f^s_l\nu^s_l(f) + T_s^{nm+j}(f) \\
     &= \frac{1}{m} \sum_{l=0}^{m-1} \sum_{j=0}^{m-1} z_m^{lj}f^s_l\nu^s_l(f) + T_s^{nm+j}(f).
\end{align*}
Now, for every $j \in \{1,...,m-1\}$, since all the $z_m^j$ are $m$-th unit roots different from $1$,
$$ \sum_{j=0}^{m-1} z_m^{lj} = 0.$$
For $j=0$, $$ \sum_{j=0}^{m-1} z_m^{0} = m.$$
Now, let us recall that $f_0^s=\mathbb{1}$. We finally obtain that:
\begin{equation}\label{eq:EgaliteSommePeriodenu0s}
    \frac{1}{m} \sum_{j=0}^{m-1} Q_s^{nm+j} f = \nu_0^s(f)+ \frac{1}{m} \sum_{j=0}^{m-1} T_s^{nm+j} (f).
\end{equation}
It remains just to prove that $\nu_0^s(f)=\eta^s(f)$. For this purpose, we recall that $\eta^s$ is invariant by $Q_s$. So by applying $\eta^s$ in Equation \eqref{eq:EgaliteSommePeriodenu0s} we obtain:
\begin{equation}\label{eq:EgaliteSommePeriodeetas}
    \eta^s(f)=\nu_0^s(f)+\frac{1}{m} \sum_{j=0}^{m-1} \eta^s(T_s^{nm+j} (f)).
\end{equation}
However for every $j \in \{0,...,m-1\}$, $\underset{n \to \infty}{\lim} \eta^s(T_s^{nm+j} (f)) =0$ since $\rho_\alpha(T_s)<1$. By taking the limit in Equation \eqref{eq:EgaliteSommePeriodeetas} we obtain that $\eta^s(f)=\nu_0^s(f)$ and the lemma is proved.
\end{proof}
For every $f,g \in \mathrm{Lip}_b(\Omega,\rr)$, we approximate them by a truncation procedure.
Let $b \in \Omega$. For every $g \in \mathrm{Lip}_b(\Omega,\rr)$ and $p \in \nn^*$, we define the function $g_p : \Omega \to \rr$ as follows, for every $\omega=(a_1,...,a_p,...) \in \Omega$,
$$g_p(\omega)=g(a_1,...,a_p,b_1,b_2,...).$$
Since $g_p(\omega)$ only depends on $(a_1,...,a_p)$, by abusing the notation, we will write $g_p(a_1,...,a_p)$ in the following. Let us now define the function $G_p : \sta \to \rr$ as follows: for every $\hat x \in \sta$,
$$ G_p(\hat x) = \int_{\mathcal A^p} g_p(a_{1},...,a_{p}) q_p^s(W_p, \hat x) d\mu^{\otimes p}(a_1,...,a_p).$$

\begin{lemma}\label{lem:GpHolder}
For every $p \in \nn^*$, $g \in \mathrm{Lip}_b(\Omega,\rr)$, the function $G_p : \sta \to \rr$ is an $\alpha$-Hölder function. Moreover there exists $C$ depending only on $s,e_s,\alpha$ and $\eta^s$ such that:
\begin{equation}
    m_\alpha(G_p) \leq C \|g\|_\infty.
\end{equation}
\end{lemma}
\begin{proof}
Let $p \in \nn^*$, $g \in \mathrm{Lip}_b(\Omega,\rr)$.  Let us first recall that for every $\hat x \in \sta$,
$$q_p^s(W_p, \hat x) = \frac{1}{k(s)^p} \frac{1}{e_s(\hat x)} e_s(W_p \cdot \hat x) \|W_p x\|^s. $$
Since $e_s$ is a strictly positive $\alpha$-Hölder function, $1/e_s$ is also a $\alpha$-Hölder function. Moreover, Lemma \ref{lem:Inegalitef_A} ensures that $f_{W_p} :\hat x \mapsto e_s(W_p \cdot \hat x) \|W_p x\|^s$ is a $\alpha$-Hölder function. Moreover there exists $C_\alpha$ such that:
\begin{equation}
    \|f_{W_p}\|_\alpha \leq C_\alpha \|W_p\|^s \|e_s\|_\alpha.
\end{equation}
And then, by submultiplicativy of the norm:
\begin{equation}
    \|q_p^s(W_p, \cdot)\|_\alpha \leq \frac{1}{k(s)^p} C_\alpha \|e_s\|_\alpha \|\tfrac{1}{e_s}\|_\alpha \|W_p\|^s.
\end{equation}
We let $C= C_\alpha \|e_s\|_\alpha \|\tfrac{1}{e_s}\|_\alpha$ depending only on $s,\alpha$ and $e_s$.
It follows that for every $\hat x, \hat y \in \sta$,
\begin{align}
    |G_p(\hat x)-G_p(\hat y)| \leq C \|g\|_\infty \frac{1}{k(s)^p} \int_{\mathcal A^p} \|W_p\|^s d \mu^{\otimes p} d(\hat x, \hat y)^\alpha .
\end{align}
Now Lemma \ref{Lemma:ExistenceSigma} implies that $\frac{1}{k(s)^p} \int_{\mathcal A^p} \|W_p\|^s d \mu^{\otimes p}$ is bounded by a constant only depending on $\eta^s$. And we finally obtain that:
$$ m_\alpha(G_p) \leq C \|g\|_\infty,$$
where $C$ depends on $s,e_s,\alpha$ and $\eta^s$.

\end{proof}

Since $G_p$ is a bounded $\alpha$-Hölder function, our goal is to use Lemma \ref{lem:SommePeriode} to prove the following lemma.
\begin{lemma}\label{lem:Inegalitefpmgpm}
Let $s \in I_\mu$ and $p \in \nn^*$. There exists $0<\lambda<1$ such that for every $n \geq p+1$ and $f,g \in \mathrm{Lip}_b(\Omega,\rr)$,
\begin{equation}
    \left | \frac{1}{m} \sum_{j=0}^{m-1} \ee_{\qq^s}[f_{pm} g_{pm} \circ \theta^{nm+j}]- \ee_{\qq^s}[f_{pm}] \ee_{\qq^s}[g_{pm}] \right | \leq \|f\|_\infty \|g\|_\infty \lambda^{nm-pm}.
\end{equation}
\end{lemma}
\begin{proof}
Let $s \in I_\mu$ and $p \in \nn^*$. Let $n \geq p+1$ and $f,g \in \mathrm{Lip}_b(\Omega,\rr)$. For every $j \in \{0,...,m-1\}$,
$f_{pm}$ depends only on $(a_1,...,a_{pm})$ and $g_{pm} \circ \theta^{nm+j}$ on $(a_{nm+j+1},...,a_{nm+j+pm})$. For clarity, we do not specify the arguments of the functions in the sequel. We also recall that for every $n \in \nn^*$, we denote by $W_n$ the random variable defined by the relation:
$$ W_n(\omega)=v_{a_n}...v_{a_1}.$$
\begin{align*}
\ee_{\qq^s}[f_{pm} g_{pm} \circ \theta^{nm+j}] &=\int_{\sta} \int_{\mathcal A^{nm+j}}f_{pm}g_{pm}q_{nm+j}^s(W_{nm+j}, \hat{x}) d \mu^{\otimes nm+j+pm} d\eta^s(\hat x) 
\end{align*}
By definition of the functions $q_n^s$, we have:
$$q_{nm+j}^s(W_{nm+j}, \hat{x})=q_{pm}^s(W_{pm}, \hat{x}) \times q_{nm-pm+j}^s(W_{pm+1}^{nm+j}, W_{pm} \cdot \hat x) \times q_{pm}(W_{nm+j+1}^{nm+j+pm}, W_{nm+j} \cdot \hat x),$$
where:
$$W_{pm+1}^{nm+j}=v_{nm +j}...v_{pm+1} \quad \mbox{and } W_{nm+j+1}^{nm+j+pm}=v_{nm+j+pm}...v_{nm+j+1}.$$
It follows that:
\begin{align*}
\ee_{\qq^s}[f_{pm} g_{pm} \circ \theta^{nm+j}] &=\int_{\sta} \int_{\mathcal A^{pm}}\int_{\mathcal A^{nm-pm+j}}\int_{\mathcal A^{pm}}f_{pm}g_{pm}q_{pm}^s(W_{pm}, \hat{x}) \\ &q_{nm-pm+j}^s(W_{pm+1}^{nm+j}, W_{pm} \cdot \hat x)  q_{pm}(W_{nm+j+1}^{nm+j+pm}, W_{nm+j} \cdot \hat x) d \mu^{\otimes nm+j+pm} d\eta^s(\hat x).
\end{align*}
Now let us notice that:
\begin{equation}
\int_{\mathcal A^{pm}} g_{pm} q_{pm}(W_{nm+j+1}^{nm+j+pm}, W_{nm+j} \cdot \hat x) d \mu^{\otimes pm} = G_{pm}(W_{nm+j}\cdot \hat{x}).
\end{equation}
Moreover, by definition of the operator $Q_s$
$$\int_{\mathcal A^{nm-pm+j}} G_{pm}(W_{nm+j}\cdot \hat{x}) q_{nm-pm+j}^s(W_{pm+1}^{nm+j}, W_{pm} \cdot \hat x) d\mu ^{\otimes nm-pm+j}= Q_s^{(n-p)m+j}G_{pm}(W_{pm} \cdot \hat x).  $$
We finally obtain that:
\begin{equation}\label{eq:Efpgpshift}
\ee_{\qq^s}[f_{pm} g_{pm} \circ \theta^{nm+j}] = \int_{\sta} \int_{\mathcal A^{pm}} f_{pm} q_{pm}^s(W_{pm}, \hat x) Q_s^{(n-p)m+j}G_{pm}(W_{pm} \cdot \hat x) d \mu ^{\otimes pm} d \eta^s (\hat x).
\end{equation}
In the other hand,
\begin{equation}\label{eq:EfpEqp}
\ee_{\qq^s}[f_{pm}] \ee_{\qq^s}[g_{pm}]=\int_{\mathcal A^{pm}} f_{pm} q_{pm}^s(W_{pm}, \hat x) \eta^s(G_{pm}) d \mu ^{\otimes pm} d \eta^s (\hat x).
\end{equation}
Now, Lemma \ref{lem:GpHolder} ensures that $G_{pm}$ is an $\alpha$-Hölder function, we can then apply Lemma \ref{lem:SommePeriode} to $G_{pm}$ and $n-p$ and we obtain that:
\begin{equation*}
    \frac{1}{m} \sum_{j=0}^{m-1} Q_s^{nm-pm+j} G_{pm} = \eta^s(G_{pm})+ \frac{1}{m} \sum_{j=0}^{m-1} T_s^{nm+j-pm} (G_{pm}).
\end{equation*}
Let us recall that $\rho_\alpha(T_s)<1$. There exists $n_0 \in \nn^*$ annd $0<\lambda <1$ such that for every $n \geq n_0$, $\|T_s^{nm-pm}\|_\infty \leq \lambda^{nm-pm}$. For $n$ large enough, we then obtain:
$$ \left \| \frac{1}{m} \sum_{j=0}^{m-1} T_s^{nm+j-pm} (G_{pm}) \right \|_\alpha \leq \|G_{pm}\|_\alpha \lambda^{nm-pm}.$$
Now, by combining Equation \eqref{eq:Efpgpshift} and Equation \eqref{eq:EfpEqp}, we obtain that,
\begin{align*}
\Bigg| \frac{1}{m} \sum_{j=0}^{m-1} &\ee_{\qq^s}[f_{pm} g_{pm} \circ \theta^{nm+j}]- \ee_{\qq^s}[f_{pm}] \ee_{\qq^s}[g_{pm}] \Bigg | \\
&\leq \int_{\sta \times \mathcal A^{pm}} f_{pm} q_{pm}^s(W_{pm}, \hat x) \left \| \frac{1}{m} \sum_{j=0}^{m-1} T_s^{nm+j-pm} (G_{pm}) \right \|_\alpha d \mu ^{\otimes pm} d \eta^s (\hat x) \\
&\leq \|f\|_\infty \|g\|_\infty \lambda^{nm-pm}.
\end{align*}
For the last inequality we used the fact that $\|G_{pm}\|_\infty \leq \|g\|_\infty$ and $\|f_{pm}\|_\infty \leq \|f\|\infty$. Indeed,
for every $r \in \nn^*$ and $\hat{x} \in \sta$:
$$ |G_r(\hat x)| \leq \int_{\mathcal A^r} \|g_r\|_\infty q_r^s(W_r, \hat x) d \mu^{\otimes r} \leq \|g_r\|_\infty \leq \|g\|_\infty.$$
\end{proof}

We now have all the tools to prove Proposition \ref{prop:decay}.

\begin{proof}[Proof of Proposition \ref{prop:decay}]
Let $f,g \in \mathrm{Lip}_b(\Omega,\rr)$. We first approximate $f$ and $g$ by $f_r$ and $g_r$ for a $r$ we specify later. For every $\omega \in \Omega$, $f \in \mathrm{Lip}_b(\Omega,\rr)$ and $r \in \nn^*$,
$$ |f(\omega)-f_r(\omega)| \leq m(f) d(\omega, (\omega_1,...,\omega_r,b)) \leq m(f) 2^{-r},$$
where $m(f)$ is the Lipschitz constant of $f$. It implies the existence of $C>0$ such that:    
\begin{align}\label{eq:Approfrgr}
    \Bigg | \frac{1}{m} \sum_{j=0}^{m-1} \ee_{\qq^s}[f g \circ \theta^{nm+j}]- \ee_{\qq^s}[f] \ee_{\qq^s}[g] &- \left (  \frac{1}{m} \sum_{j=0}^{m-1} \ee_{\qq^s}[f_r g_r \circ \theta^{nm+j}]- \ee_{\qq^s}[f_r] \ee_{\qq^s}[g_r]\right ) \Bigg| \\
    &\leq C \|f\|_1 \|g\|_1 2^{-r}.
\end{align}
Now let $n \geq 2$, $p= \lfloor \frac{n}{2} \rfloor$ and $r=pm$. Lemma \ref{lem:Inegalitefpmgpm} implies the existence of $0<\lambda<1$ such that for $n$ large enough:
\begin{equation}\label{eq:UtilisationLemfrgr}
    \left |  \frac{1}{m} \sum_{j=0}^{m-1} \ee_{\qq^s}[f_r g_r \circ \theta^{nm+j}]- \ee_{\qq^s}[f_r] \ee_{\qq^s}[g_r]\right | \leq \|f\|_\alpha \|g\|_\alpha (\lambda^m)^{n-p}.
\end{equation}
We set $\Tilde{\lambda}=\max \{ 2^{-m}, \lambda^{m} \}$. We recall that $n-p=n-\lfloor \frac{n}{2} \rfloor= \lceil \frac{n}{2} \rceil$ and that for every $0 \leq t \leq 1$, the function $y \mapsto t^{y}$ is decreasing. By combining Equation \eqref{eq:Approfrgr} and Equation \eqref{eq:UtilisationLemfrgr} we finally obtain the existence of $C>0$ and $0<\Tilde{\lambda}<1$  such that:

\begin{equation}
     \left | \frac{1}{m} \sum_{j=0}^{m-1} \ee_{\qq^s}[f g \circ \theta^{nm+j}]- \ee_{\qq^s}[f] \ee_{\qq^s}[g] \right | \leq C \|f\|_1 \|g\|_1 (\Tilde{\lambda})^{\lfloor \frac{n}{2} \rfloor}.
\end{equation}
It proves the proposition.
\end{proof}

A straightforward corollary of Proposition \ref{prop:decay} is the ergodicity of the probability measure $\qq^s$.

\begin{corollary}\label{Cor:ErgodicShift}
Assume that {\bf (Cont)} and {\bf (Irr)} hold. Then, for every $s \in I_\mu$, the probability measure $\qq^s$ is ergodic for the shift $\theta$.
\end{corollary}

\subsection{Variational Principle}\label{Sec:VariationalPrinciple}
In this section we are keeping the notations of Section \ref{sec:Decay}. The goal of this section is to link the previous results with the thermodynamic formalism, more specifically with the results of Feng, Käenmanki and Morris \cite{morris2018ergodic,Fe11}.
In the following, we assume that $\mu$ is a probability measure on $\mathcal A$.
We first define the \textit{pressure} function $P$ of the map $v:\mathcal A \to \mat$. For every $s \in \rr_+$:
$$P(s)=\underset{n \to \infty}{\lim} \frac{1}{n} \log \left( \int_{\mathcal A^n} \|v_{a_n}...v_{a_1}\|^s d\mu^{\otimes n}(a_1,...,a_n) \right)=\log k(s).$$
For every $\pp, \qq \in \mathcal P_\theta$ and $n \in \nn$ we denote by $S(\pp_n|\mu^{\otimes n})$ the  relative entropy of $\pp_n$ with respect to $\mu^{\otimes n}$. By the variational principle \cite{donsker1983asymptotic} we have:
$$ S(\pp_n|\mu^{\otimes n}) = \underset{f : \mathcal A^n \to \rr, f \; \text{bounded}}{\sup}(\ee_{\pp}[f]- \log \ee_{\mu^{\otimes n}}[e^f]).$$

\begin{lemma}\label{lemma:ExistenceLimiteEntropie}
For every $\pp \in \mathcal P_\theta$, the sequence $(-\frac{1}{n}S(\pp_n|\mu^{\otimes n}))$ converges. Moreover, its limit $h_\mu(\pp)$ satisfies the relation:
$$ h_\mu(\pp) = \underset{n \in \nn}{\inf} -\frac{1}{n}S(\pp_n|\mu^{\otimes n}) \in [-\infty,0].$$
\end{lemma}
\begin{proof}
Let $\pp \in \mathcal P_\theta$, $n,m \in \nn$ and $f_n : \mathcal A^n \to \rr$ and $f_m : \mathcal A^m \to \rr$ be two measurable functions. Let us define the function $f_{n+m} : \mathcal A^{n+m}$ as follows: for every $(a_1,...,a_n,a_{n+1},...,a_{n+m}) \in \mathcal A^{n+m}$,
$$ f_{n+m}(a_1,...,a_n,a_{n+1},...,a_{n+m})=f_n(a_1,...,a_n)+f_m(a_{n+1},...,a_{n+m}).$$
By linearity of the expectation and since $\mu^{\otimes \nn}$ is a product measure, by independence of $f_n$ and $f_m \circ \theta^n$ we have:
\begin{align*}
\ee_{\pp}[f_{n+m}]- \log \ee_{\mu^{\otimes \nn}}[e^{f_{n+m}}] &= \ee_{\pp}[f_n]+\ee_{\pp}[f_m] - \log \ee_{\mu^\otimes \nn}[e^{f_n+f_m}] \\
&= \ee_{\pp}[f_n]-\log \ee_{\mu^{\otimes \nn}}[e^{f_{n}}])+\ee_{\pp}[f_m] - \log \ee_{\mu^{\otimes \nn}}[e^{f_{m}}]).
\end{align*}
By definition of $S(\pp_n|\mu^{\otimes n})$ and $S(\pp_m|\mu^{\otimes m})$, it follows that:
\begin{align*}
     S(\pp_n|\mu^{\otimes n}) + S(\pp_m|\mu^{\otimes m}) &\leq \underset{f_n : \mathcal A^n \to \rr}{\sup} \underset{f_m : \mathcal A^m \to \rr}{\sup} \ee_{\pp}[f_{n+m}]- \log \ee_{\mu^{\otimes \nn}}[e^{f_{n+m}}] \\
    &\leq \underset{f_{n+m}: \mathcal A^{n+m} \to \rr}{\sup} \ee_{\pp}[f_{n+m}]- \log \ee_{\mu^{\otimes \nn}}[e^{f_{n+m}}] \\
    &\leq S(\pp_{n+m}|\mu^{\otimes n+m}).
\end{align*}
It follows that the sequence $(-S(\pp_n|\mu^{\otimes n}))$ is subadditive. By Fekete lemma, the sequence $(-\frac{1}{n}S(\pp_n|\mu^{\otimes n}))$ converges to $\underset{n \in \nn^*}{\inf} -\frac{1}{n}S(\pp_n|\mu^{\otimes n})$. By positivity of the relative entropy, its limit $h_\mu(\pp)$ is non-positive.
\end{proof}

To define the next object, we add a last assumption: there exists $c>0$ such that for every $a \in \mathcal A$, 
$$ \|v_a\| \leq c.$$
It implies that for every $(a_1,...,a_n) \in \mathcal A^n$,
\begin{equation}\label{eq:HypEnergy}
-\frac{1}{n} \log \|v_{a_n}...v_{a_1}\| \geq -c.
\end{equation}
For every $\pp \in \mathcal P_\theta$, we define the \textit{specific energy} of $\pp$:
\begin{equation}
    \zeta(\pp)=\underset{n \to \infty}{\lim} \ee_{\pp}\left[-\frac{1}{n} \log \|W_n\| \right],
\end{equation}
where $W_n: \Omega \to \mat$ is the random variable defined by $W_n(a_1,...,)=v_{a_n}...v_{a_1}$. Equation \eqref{eq:HypEnergy} and Fekete Lemma ensure that this quantity is well defined for every $\pp \in \mathcal P_\theta$. The limit may be $+ \infty$.
We now are able to state the main result of this section.

\begin{theorem}[Variational Principle]\label{Th:VariationalPrinciple}
For every $s \in I_\mu$,
\begin{enumerate}
    \item The pressure function verifies 
    \begin{equation}\label{eq:variationalprinciple}
        P(s)=\underset{\pp \in \mathcal P_\theta}{\sup} (h_\mu(\pp)-s \zeta(\pp)).
    \end{equation}
    We define the set of \textit{equilibrium states} $\mathcal P_{eq}(s)=\{ \pp \in \mathcal P_\theta\; | \; P(s)=(h_\mu(\pp)-s \zeta(\pp)) \} $.
    \item The probability measure $\qq^s=\int_{\sta} \qq_x^s d \eta(\hat{x})$ is ergodic for the left shift $\theta:\mathcal{A}^\nn \mapsto \mathcal{A}^\nn$ and saturates Equation \eqref{eq:variationalprinciple}, i.e. $\qq^s \in \mathcal P_{eq}(s)$. Moreover, it is the only element of $\mathcal P_{eq}(s)$, $\mathcal P_{eq}(s) =\{ \qq^s \}$ and then for every $s \in I_\mu$,
    \begin{equation}
        P(s)=h_\mu(\qq^s)-s\zeta(\qq^s).
    \end{equation}
\end{enumerate}
\end{theorem}
In order to prove Theorem \ref{Th:VariationalPrinciple}, we state two useful properties of the probability measure $\qq^s$. In the sequel, for every $n \in \nn$, we identify $q^s_n(a_1,...,a_n)=\int_{\sta} q^s(\hat{x},a_1,...,a_n) d\eta^s(\hat{x})$ with $q^s_n(W_n)$ the function defined in Equation \eqref{eq:Defqs}.
\begin{lemma}\label{lemma:propertiesqs}
For every $s \in I_\mu$, there exists $A_s>0$ and $B_s>0$ such that for $n \in \nn$, and $W \in \mat$
\begin{equation}
A_s \frac{1}{k(s)^n}\|W\|^s \leq q_n^s(W) \leq B_s \frac{1}{k(s)^n}\|W\|^s.
\end{equation}
\end{lemma}

\begin{proof}
Let $s \in I_\mu$. We recall that $q_n^s=\frac{1}{k(s)^n} J_s$. By applying Lemma \ref{lemma:ProprietesIs}, we obtain the existence of $A_s>0$ such that for every $W\in \mat$,
$$A_s \frac{1}{k(s)^n}\|W\|^s \leq q_n^s(W).$$
For the second inequality, we set $B_s=\underset{\hat x,\hat y \in \sta}{\sup} \frac{e_s(\hat x)}{e_s(\hat y)}>0$. It follows that for every $W \in \mat$, and $\hat x \in \sta$
$$\frac{e_s(W \cdot \hat x)}{e_s(\hat x)}\|Wx\|^s \leq B_s \|W\|^s.$$
Since $\eta^s$ is a probability measure on $\sta$, we obtain the second inequality.
\end{proof}

\begin{lemma}\label{lemma:submultipliqqs}
Let $s \in I_\mu$. There exists $C_s >0$ such that for every $n,m \in \nn$, $a \in \mathcal{A}^n$, $b \in \mathcal A^m$
$$q^s_{n+m}(W_mW_n) \leq C_s q^s_n(W_n) q^s_m(W_m).$$
\end{lemma}

\begin{proof}
Let $s \in I_\mu$, $n,m \in \nn$ and $a \in \mathcal{A}^n$, $b \in \mathcal A^m$.
By definition, $q^s_{n+m}(W_mW_n)=q^s_{n+m}(v_{b_m}...v_{b_1}v_{a_n}...v_{a_1})$. Then,
$$q^s_{n+m}(W_mW_n)=\frac{1}{k(s)^{n+m}}\int_{\sta} \frac{e_s(W_mW_n \cdot \hat x)}{e_s(\hat x)}\|W_mW_nx\|^s d \eta^s(\hat x).$$
On one hand, since $e_s$ is a strictly positive bounded function, we have:
$$\frac{e_s(W_mW_n \cdot \hat x)}{e_s(\hat x)}=\frac{e_s(W_n \cdot \hat x)}{e_s(\hat x)} \frac{e_s(W_mW_n \cdot \hat x)}{e_s(W_n\cdot \hat x)} \leq c_s \frac{e_s(W_n \cdot \hat x)}{e_s(\hat x)},$$
where $c_s=\underset{\hat x, \hat y \in \sta}{\sup} \frac{e_s(\hat x)}{e_s(\hat y)} >0$. On the other hand, for every $\hat x \in \sta$,
$$ \|W_m W_n x\|^s \leq \|W_m\|^s \|W_n x\|^s.$$
By definition of the functions $q^s_n$,
$$ q^s_{n+m}(W_mW_n) \leq c_s q^s_n(W_n) \times \frac{1}{k(s)^m} \|W_m\|^s.$$
Now, by Lemma \ref{lemma:propertiesqs}, there exists $K_s>0$ such that:
$$ \frac{1}{k(s)^m} \|W_m\|^s \leq K_s q^s_m(W_m).$$
By setting $C_s=c_s K_s$, we obtain the desired inequality.
\end{proof}
In order to prove Theorem \ref{Th:VariationalPrinciple}, we need a final definition.
\begin{lemma}
For every $\pp \in \mathcal P_\theta$, the following quantity is well defined:
$$\xi_s(\pp)=\underset{n \to \infty}{\lim} \ee_{\pp}\left[-\frac{1}{n} \log q_n^s(W_n) \right].$$
\end{lemma}

\begin{proof}
Let $\pp \in \mathcal P_\theta$. Assumption \eqref{eq:HypEnergy} and Lemma \ref{lemma:propertiesqs} ensures that for every $n \in \nn^*$, $\ee_{\pp}\left[-\frac{1}{n} \log q_n^s(W_n) \right]$ is well defined. Now, Lemma \ref{lemma:submultipliqqs} implies that there exists $C>0$ such that for every $n,m \in \nn^*$, $q^s_{n+m}(W_{m+n}) \leq C q^s_n(W_n) q^s_m(W_m \circ \theta^n).$ It follows that $\log q^s_{n+m}(W_{m+n} \leq \log q^s_n(W_n) + \log q^s_m(W_m \circ \theta^n) +\log C$. Then, since $\pp$ is $\theta$-invariant,
\begin{equation}
    \log C + \ee_{\pp}[\log q_{n+m}^s(W_{n+m})] \leq (\log C + \ee_{\pp}[\log q_{n}^s(W_{n})]) + (\log C +\ee_{\pp}[\log q_{m}^s(W_{m})]).
\end{equation}
The sequence $(\log C + \ee_{\pp}[\log q_{n}^s(W_{n})])$ is subadditive, by Fekete Lemma, it follows that $\underset{n \to \infty}{\lim} \ee_{\pp}\left[-\frac{1}{n} \log q_n^s(W_n) \right]$ exists.
\end{proof}

This lemma allows tu prove the following proposition.
\begin{proposition}\label{Prop:EgalitePs}
For every $\pp \in \mathcal{P}_\theta$ and $s \in I_\mu$,
\begin{equation}\label{eq:DefinitionXis}
    \xi_s(\pp)-s \zeta(\pp) = P(s).
\end{equation}
\end{proposition}

\begin{proof}
Let $s \in I_\mu$ and $\pp \in \mathcal P(\Omega)$. Let $n \in \nn$,
\begin{align*}
     -s\ee_{\pp}\left[-\frac{1}{n} \log \|W_n\| \right] + \ee_{\pp}\left[-\frac{1}{n} \log q_n^s(W_n) \right] = \ee_{\pp}\left[\frac{1}{n} \log \frac{\|W_n\|^s}{q_n^s(W_n)}\right].
\end{align*}
Lemma \ref{lemma:propertiesqs} implies the existence of $A_s>0$ and $B_s>0$ such that:
$$ \frac{k(s)^n}{A_s} \leq \frac{\|W_n\|^s}{q_n^s(W_n)} \leq \frac{k(s)^n}{B_s}.$$
It implies that:
\begin{equation}
    \ee_{\pp}\left[ \log k(s) - \frac{1}{n}\log A_s \right] \leq \ee_{\pp}\left[\frac{1}{n} \log \frac{\|W_n\|^s}{q_n^s(W_n)}\right] \leq \ee_{\pp}\left[ \log k(s) - \frac{1}{n}\log B_s \right].
\end{equation}
By taking the limit when $n$ goes to the infinity:
$$ \xi_s(\pp)-s \zeta(\pp) = P(s).$$
\end{proof}
Let us now prove Theorem \ref{Th:VariationalPrinciple}.
\begin{proof}[Proof of Theorem \ref{Th:VariationalPrinciple}]
Let $s \in I_\mu$. We first prove that for every $\pp \in \mathcal P_\theta$,
$$ h_\mu(\pp) - s \zeta(\pp) \leq P(s). $$
Let $\pp \in \mathcal P_\theta$. If $h_\mu(\pp)=-\infty$, the inequality is trivial. We then suppose that $h_\mu(\pp)=-M$, with $M \geq 0$. It implies that for every $n \in \nn$,
$$ S(\pp_n | \mu^{\otimes n}) \leq Mn.$$
It follows that $\pp_n \ll \mu^{\otimes n}$ for every $n \in \nn$. In the sequel, for every $n \in \nn$, we denote by $f_n$ the Radon-Nykodim derivative $f_n= \frac{d \pp_n}{d \mu^{\otimes n}}$. we can then write:
\begin{equation}\label{eq:Sppnfn}
    S(\pp_n | \mu^{\otimes n})=\int_{\mathcal A^n} \log f_n(a_1,...,a_n) \d \pp_n(a_1,...,a_n)=\ee_\pp[\log f_n].
\end{equation}
It follows that for every $n \in \nn$,
\begin{equation}\label{eq:EgaliteSppn}
-\frac{1}{n}S(\pp_n|\mu^{\otimes n})-s \ee_\pp \left[ -\frac{1}{n} \log \|W_n\| \right ]= \frac{1}{n} \ee_\pp\left[\log \frac{\|W_n\|^s}{f_n}\right].   \end{equation}
By Jensen Lemma, the concavity of the logarithm yields:
\begin{align*}
    \ee_\pp\left[\log \frac{\|W_n\|^s}{f_n}\right] &\leq \log \ee_\pp \left[\frac{\|W_n\|^s}{f_n}\right] \\
    &= \log \left(\int_{\mathcal A^n} \frac{\|W_n\|^s}{f_n}f_n d\mu^{\otimes n} \right)\\
    &= \log \int_{\mathcal A^n} \|W_n\|^s d \mu^{\otimes n}.
\end{align*}
Combining this inequality with Equation \eqref{eq:EgaliteSppn} gives:
$$-\frac{1}{n}S(\pp_n|\mu^{\otimes n})-s \ee_\pp \left[ -\frac{1}{n} \log \|W_n\| \right ] \leq \frac{1}{n} \log \int_{\mathcal A^n} \|W_n\|^s d \mu^{\otimes n}.$$
By taking the limit when $n$ goes to infinity, we obtain:
$$ h_\mu(\pp) -s \zeta(\pp) \leq P(s).$$
Let us now prove Item $(2)$. The ergodicity of $\qq^s$ has been proved in Corollary \ref{Cor:ErgodicShift}.
Now, let us prove that $h_\mu(\qq^s) -s \zeta(\qq^s) = P(s)$. We first recall that for every $n \in \nn$, $\qq^s_n \ll \mu^{\otimes n}$ and $q_n^s(W_n)= \frac{d \qq^s_n}{d \mu^{\otimes n}}$. Equation \eqref{eq:Sppnfn} applied to $\qq^s$ gives:
$$S(\qq^s_n|\mu^{\otimes^n})= \ee_{\qq^s}[\log q_n^s(W_n)].$$
By definition of $\xi_s$ \eqref{eq:DefinitionXis},
$$ \underset{n \to \infty}{\lim} -\frac{1}{n}S(\qq^s_n|\mu^{\otimes^n}) -s \ee_{\qq^s} \left[ -\frac{1}{n} \log \|W_n\| \right ]= \xi_s(\qq^s)-s \zeta(\qq^s)=P(s).$$
The last equality is given by Proposition \ref{Prop:EgalitePs}. It remains to prove that $\mathcal P_{eq}(s)=\{\qq^s\}$. Let $\pp \in \mathcal P_{eq}(s)$.
Since $\pp \in \mathcal P_{eq}(s)$, $h_\mu(\pp) > - \infty$ and then for every $n \in \nn$, $\pp_n \ll \mu^{\otimes n}$. Moreover, we already know that for every $n \in \nn$, $\qq^s_n \ll \mu^{\otimes n}$. It implies that for every $n \in \nn$:
\begin{equation}\label{eq:Sppnqqsn}
    S(\pp_n|\qq^s_n)=\ee_\pp\left[ \log \frac{f_n}{q^s_n(W_n)} \right ]=\ee_\pp[\log f_n ]-\ee_\pp[\log q^s_n(W_n)],
\end{equation}
where $f_n = \frac{d \pp_n}{d \mu^{\otimes n}}$. Now let us recall that $\ee_\pp[\log f_n ]=S(\pp_n|\mu^{\otimes n})$ \eqref{eq:Sppnfn}.
It follows that for every $n \in \nn$:
\begin{align*}
    -\frac{1}{n}S(\pp_n|\qq^s_n) = &-\frac{1}{n} S(\pp_n|\mu^{\otimes n})-s \ee_\pp \left[ -\frac{1}{n} \log \|W_n\| \right] \\
    &-  \left(\ee_\pp[\log q^s_n(W_n)]-s\ee_\pp \left[ -\frac{1}{n} \log \|W_n\| \right]\right).
\end{align*}
Now by taking the limit when $n$ goes to infinity we obtain:
\begin{align*}
    \underset{n \to \infty}{\lim}-\frac{1}{n}S(\pp_n|\qq^s_n) = h_\mu(\pp)-s \zeta(\pp) - ( \xi_s(\pp)-s\zeta(\pp)).
\end{align*}
By assumption, $h_\mu(\pp)-s \zeta(\pp)=P(s)$ and Proposition \ref{Prop:EgalitePs} implies that $\xi_s(\pp)-s\zeta(\pp)=P(s)$.
We finally obtain that:
\begin{align*}
    \underset{n \to \infty}{\lim}-\frac{1}{n}S(\pp_n|\qq^s_n) = 0.
\end{align*}
Equation \eqref{eq:Sppnfn} and Equation \eqref{eq:Sppnqqsn} imply that for every $n \in \nn$.
$$ S(\pp_n|\qq^s_n)=S(\pp_n|\mu^{\otimes n})-\ee_\pp[\log q^s_n(W_n)].$$
Now, let $n,m \in \nn$, Lemma \ref{lemma:submultipliqqs} gives the existence of $C_s >0$ such that:
$$\ee_\pp[\log q^s_{n+m}(W_{n+m})] \leq \log C_s + \ee_\pp[\log q^s_{m}(W_{m})] + \ee_\pp[\log q^s_{n}(W_{n})].$$
Moreoever the subadditivity of the sequence $(-S(\pp_n|\mu^{\otimes n}))$ has been proven in the proof of Lemma \ref{lemma:ExistenceLimiteEntropie}.
It follows that for every $n_m \in \nn$,
$$ -S(\pp_{n+m}|\qq^s_{n+m}) \leq -S(\pp_{n}|\qq^s_{n}) -S(\pp_{m}|\qq^s_{m})+\log C_s. $$
The sequence $(-S(\pp_{n}|\qq^s_{n})+\log C_s)$ is subadditive, by Fekete Lemma,
$$ \underset{n \to \infty}{\lim} \frac{-S(\pp_{n}|\qq^s_{n})+\log C_s}{n}=\underset{n \in \nn}{\inf} \frac{-S(\pp_{n}|\qq^s_{n})+\log C_s}{n}.$$
However,
$$ \underset{n \to \infty}{\lim} \frac{-S(\pp_{n}|\qq^s_{n})+\log C_s}{n}= \underset{n \to \infty}{\lim} \frac{-S(\pp_{n}|\qq^s_{n})}{n}=0.$$
It implies that for every $n \in \nn$,
$$ S(\pp_n|\qq^s_n) \leq \log C_s.$$
The lower semicontinuity of the relative entropy implies that:
$$ S(\pp|\qq^s) \leq 
\underset{n \to \infty}{\liminf} S(\pp_n|\qq^s_n) \leq \log C_s. $$
We then obtain that $\pp \ll \qq^s$. Since $\qq^s$ is $\theta$-ergodic, $\pp=\qq^s$ and the theorem is proved.
\end{proof}

\section{Examples} \label{sec:Examples}
Note that our assumptions require that $s_->0$ and one can wonder if we can extend it to $s=0$. We provide an example in the first subsection that shows $s_-$ is necessary strictly positive. The second example \cite[Example 3.5]{Fe02} shows that the irreducibility assumption is crucial for the analyticity of the pressure.
\subsection{Counter example in $s=0$}
In this section, we are interested in the Keep-Switch \textit{Positive Matrix Product} instrument that has been studied in Section 2.1 of \cite{BCJP21}. We use the formalism of Section 2.1.2 in \cite{BCJP21}. Let $q_1,q_2 \in ]0,1[$ such that $q_1 \neq q_2$ and $r_1=1-q_1, r_2=1-q_2$. Let $\mathcal{A}=\{K,S\}$ an alphabet with two elements and $\Omega=\mathcal{A}^\nn$. We keep the notations of Section \ref{Sec:VariationalPrinciple} and $\mathcal O$ denotes the $\sigma$-algebra defined in this section. We define the matrices $M_K$ and $M_S$ as follows:
\begin{equation}
    M_K=\left(
\begin{array}{ll}
        q_1 & 0\\
        0 & q_2
    \end{array}
\right) \; , \;
M_S=\left(
\begin{array}{ll}
        0 & r_1\\
        r_2 & 0
    \end{array}
\right).
\end{equation}
Let $\textbf{p}:=(r_1+r_2)^{-1}[r_2 \; r_1]$. For every, $n \in \nn$, we define the probability measure $\pp_n$ on $(\mathcal A^n,\mathcal O_n)$ such that for every $\omega \in \mathcal A^\nn$,
$$ \pp_n(\omega)= \textbf{p}M_{\omega_1}...M_{\omega_n} \textbf{1},$$
where $\textbf{1}=[1 \; 1]^T$ and $T$ denotes the transposition. 
As explained in Section 2.1 \cite{BCJP21}, there exists a unique probability measure $\pp \in \mathcal{P}_\theta$ such that for every $n \in \nn^*$ and $\omega \in \mathcal A^n$:
$$ \pp[\{ \Bar{\omega} \in \Omega \; | \; (\Bar{\omega}_1,...,\Bar{\omega}_n) = \omega \}]=\pp_n(\omega).$$
It is called the positive matrix product measure generated by $((M_K,M_S),\textbf{p})$. Let us prove that the set $\{M_K,M_S\}$ verifies the assumptions {\bf (Irr)} and {\bf (Cont)}. We consider a particular sequence $(W_n)$ defined by the relation 
$$W_n = M_K^n =\left(
\begin{array}{ll}
        q_1^n & 0\\
        0 & q_2^n
    \end{array}
\right) .$$
Without loss of generality, we can assume that $q_1>q_2$. It follows that $\|W_n\|_2=q_1^n$, and then $$\frac{W_n}{\|W_n\|}=\left(
\begin{array}{ll}
        1& 0\\
        0 & \left (\frac{q_2}{q_1} \right)^n
    \end{array}
\right).$$
Since $q_1>q_2$, the sequence $\left (\frac{W_n}{\|W_n\|} \right)$ converges towards $\left(
\begin{array}{ll}
        1& 0\\
        0 & 0
    \end{array}
\right)$ a rank-one matrix. That proves {\bf (Cont)}. The matrices $M_K$ and $M_S$ have a non-trivial common invariant subspace if and only if they share a common eigenvector. However, the eigenvectors of $M_K$ are multiples of $\left(\begin{array}{l}
        1\\
        0 
    \end{array}\right)$ and 
$\left(\begin{array}{l}
        0\\
        1
    \end{array}\right)$.
These two vectors are not eigenvalues of $M_K$, it implies that Assumption {\bf (Irr)} is verified. In Remark 2.30 of \cite{BCJP21}, this specific example has been studied. Indeed, for this example, for every $\omega \in \mathcal A^n$, $\pp_n(\omega) \sim \|M_{\omega_n}...M_{\omega_1}\|$. In the context of \cite{BCJP21}, the variable is not $s$ but $\alpha$ and the analog of $\log k(s)$ is $r(\alpha)$ representing the Rényi entropy.
When $q_1 \neq q_2$, the pressure function is only analytic on $\rr \backslash \{0\}$ and not twice differentiable at $0$. This example shows that under the assumptions {\bf (Irr)} and {\bf (Cont)}, the function $s \mapsto k(s)$ may not be analytic in $0$.

\subsection{Counter example beyond irreducibility}
The {\bf (Irr)} assumption is crucial for the analyticity of the function $s \mapsto \log k(s)$. The following example illustrates this fact. We endow the space $\mat$ with the norm $\|.\|_1$. Let $a,b,c,d \in \rr_+^*$,
$\mathcal{A}= \left \{ \begin{pmatrix}
a & 0\\
0 & b
\end{pmatrix}, \begin{pmatrix}
c & 0\\
0 & d
\end{pmatrix}\right \}=\{A,B\},$
and let $\mu$ be the measure on $\mat$ defined as $\mu=\delta_A + \delta_B$. The set $\mathcal{A}$ is the support of $\mu$ and does not verify the assumption {\bf (Irr)} because the matrices $A$ and $B$ share a common eigenvector $\left(\begin{array}{l}
        1\\
        0 
    \end{array}\right)$ . 
For every $s \in \rr_+^*$ and $n \in \nn$*, we defined the sequence $(Z_n(s))$ as follows:
\begin{equation}
    Z_n(s)=\sum_{v_i,...,v_n \in \mathcal{A}^{n}} \|v_n...v_1\|_1^s < \infty.
\end{equation}
A product of $n$ matrices of $\mathcal{A}$ can be written
$\begin{pmatrix}
a^lc^{n-l} & 0\\
0 & b^ld^{n-l}
\end{pmatrix} $ where $l$ is the number of occurences of $A$ in the product. It follows that:
$$ Z_n(s)=\sum_{l=0}^n \begin{pmatrix}
n\\
l 
\end{pmatrix}(a^lc^{n-l} + b^ld^{n-l})^s. $$
If $s \geq 1$, Jensen Inequality implies that for every $l=0,...,n$ $$(a^lc^{n-l} + b^ld^{n-l})^s \leq 2^s (a^{sl}c^{s(n-l)}+b^{sl}d^{s(n-l)}),$$ and the positivity of $a,b,c,d$ implies that:
\begin{equation}
    \max \left \{\sum_{l=0}^n \begin{pmatrix}
n\\
l 
\end{pmatrix}a^{sl}c^{s(n-l)}, \sum_{l=0}^n \begin{pmatrix}
n\\
l 
\end{pmatrix}b^{sl}d^{s(n-l)} \right \} \leq Z_n(s) \leq 2^{s+1} \max \left \{\sum_{l=0}^n \begin{pmatrix}
n\\
l 
\end{pmatrix}a^{sl}c^{s(n-l)}, \sum_{l=0}^n \begin{pmatrix}
n\\
l 
\end{pmatrix}b^{sl}d^{s(n-l)} \right \}.
\end{equation}
We can rewrite this equation:
\begin{equation}
    \max \{ (a^s+c^s)^n, (b^s+d^s)^n \} \leq Z_n(s) \leq 2^{s+1} \max \{ (a^s+c^s)^n, (b^s+d^s)^n \}.
\end{equation}
It follows that for every $s \geq 1$,
$$
    \log k(s) = \underset{n \to \infty}{\lim} \frac{1}{n} \log Z_n(s) = \max \{ \log(a^s+c^s), \log(b^s+d^s) \}.
$$
By doing a similar proof, we obtain the same formula for the case $0<s<1$.
We are interested in the analyticity of $s \mapsto \log k(s)$.
\begin{itemize}
    \item If the equation (in $s$) $a^s+c^s=b^s+d^s$ does not have any solutions in $\rr_+^*$, then by continuity of the function $s \mapsto a^s+c^s-b^s-d^s$ either $a^s+c^s > b^s + d^s$ or $a^s+c^s < b^s + d^s$ for every $s \in \rr_+^*$. We can assume for example that $a^s+c^s \geq b^s + d^s$ and it follows that $\log k(s)=\log(a^s+c^s)$ for every $s \in \rr_+^*$. The analyticity of $\log k$ follows. For $a=3, b=2, c=1/3, d=1/2$ it is the case. Indeed, a study of the function $s \mapsto 3^s+(1/3)^s-2^s-(1/2)^s$ shows that $3^s+(1/3)^s \geq 2^s+(1/2)^s$ for every $s \in \rr$. It follows that $\log k(s)=\log(3^s+(1/3)^s)$.
    \begin{center}
  \begin{tikzpicture}
    \begin{axis}[
      xlabel={$s$},
      ylabel={$\log k(s)$},
      domain=0:4,
      samples=100,
      grid=major,
      title={$a=3, b=2, c=1/3, d=1/2$},
      legend entries={$\ln(3^s + (1/3)^s)$},
      legend style={at={(0.5,-0.2)}, anchor=north},
    ]
      \addplot[blue] {ln(3^x + (1/3)^x)};
    \end{axis}
  \end{tikzpicture}
\end{center}
    \item If there exists a solution $s_c$ to the equation $a^s+c^s=b^s+d^s$, then if $\ln{(a)}a^s+\ln{(c)}c^s \neq \ln{(b)}b^s + \ln{(d)}d^s$, the function $\log k$ is not differentiable at this point and therefore not analytic. It is the case for $a=2,b=1,c=2,d=3$. A study of the function shows that the function $s \to 3^s+1^s-2^s-2^s$ only vanishes when $s=1$ on $\rr_+^*$. It means that the only critical point $s_c$ is equal to $1$. At this point $\log k$ is not differentiable.
    Indeed:
    $$ 2^1+2^1=3^1+1^1,$$
    and:
    $$ 2 \ln{2} + 2 \ln{2} \neq 3 \ln{3}.$$
    \begin{center}
  \begin{tikzpicture}
    \begin{axis}[
      xlabel={$s$},
      ylabel={$\log k(s)$},
      domain=0.95:1.05,
      samples=100,
      grid=major,
      title={$a=2,b=1,c=2,d=3$},
      legend style={at={(0.5,-0.2)}, anchor=north},
    ]
      \addplot[blue, domain=0.95:1] {ln(2^x + 2^x)}; 
      \addplot[red, domain=1:1.05] {ln(1 + 3^x)};    
      
      \legend{$\ln(2^s+2^s)$, $\ln(1+3^s)$}
    \end{axis}
  \end{tikzpicture}
\end{center}
\end{itemize}

\paragraph{\bf Acknowledgements} 
The author expresses his gratitude to his PhD advisors, Tristan Benoist and Clément Pellegrini, for their guidance and the many fruitful discussions.
The author was supported by the ANR project ``ESQuisses", grant number ANR-20-CE47-0014-01 and by the ANR project ``Quantum Trajectories'' grant number ANR-20-CE40-0024-01. This work received support from the University Research School EUR-MINT
(State support managed by the National Research Agency for Future Investments
 program bearing the reference ANR-18-EURE-0023)

\bibliographystyle{amsalpha}
\providecommand{\bysame}{\leavevmode\hbox to3em{\hrulefill}\thinspace}
\providecommand{\MR}{\relax\ifhmode\unskip\space\fi MR }
\providecommand{\MRhref}[2]{%
  \href{http://www.ams.org/mathscinet-getitem?mr=#1}{#2}
}
\providecommand{\href}[2]{#2}

\end{document}